\documentclass[twoside]{article}
\usepackage{mathtools}
\usepackage{latexsym}
\usepackage{mathrsfs}
\usepackage{pstricks}
\usepackage{amssymb}
\usepackage{amsmath}
\usepackage{amsfonts}
\usepackage{amsthm}
\usepackage{array}
\usepackage{epsfig}
\usepackage{chemarrow}
\usepackage{verbatim}
\usepackage{graphicx}
\usepackage{float}
\usepackage{subfigure}
\usepackage{tabularx}
\usepackage{booktabs}
\providecommand{\U}[1]{\protect\rule{.1in}{.1in}}

  \newcolumntype{Y}{>{\raggedleft\arraybackslash}X}

\def\bn{{\mathbb{N}}}

\def\br{{\mathbb{R}}}

\def\bz{{\mathbb{Z}}}

\def\br{\mathbb R}

\def\td{{\tilde d}}
\def\tD{{\tilde D}}

\def\vs{\vskip.3cm}
\def\noi{\noindent}

\def\vp{\varphi}

\def\ker{\text{\rm Ker\,}}

\DeclareMathOperator{\gl}{GL}

\DeclareMathOperator{\id}{Id}
\newcommand{\amal}[5]{#1\prescript{#4}{#5}\times_{#3}#2}

\newcommand\cV{\ensuremath{\mathcal V}}

\newrgbcolor{violet}{.6 .1 .8}
  \newrgbcolor{lightyellow}{1 1 .8}
  \newrgbcolor{lightblue}{.80 1 1}
  \newrgbcolor{mygreen}{0 .66 .05}
  \definecolor{mygreen}{rgb}{0,.66,.05}
  \definecolor{lightyellow}{rgb}{1,1,.80}
  \newrgbcolor{orange}{1 .6 0}
  \newrgbcolor{GreenYellow}{.85 1 .31}
  \newrgbcolor{Yellow}{1  1  0}
  \newrgbcolor{Goldenrod}{1  .90  .16}
  \newrgbcolor{Dandelion}{1  .71  .16}
  \newrgbcolor{Apricot}{1  .68  .48}
  \newrgbcolor{Peach}{1  .50  .30}
  \newrgbcolor{Melon}{1  .54  .50}
  \newrgbcolor{YellowOrange}{1  .58  0}
  \newrgbcolor{Orange}{1  .39  .13}
  \newrgbcolor{BurntOrange}{1  .49  0}
  \newrgbcolor{Bittersweet}{1.  .4300  .24}
  \newrgbcolor{RedOrange}{1  .23  .13}
  \newrgbcolor{Mahogany}{1.  .4475  .4345}
  \newrgbcolor{Maroon}{1.  .4084  .5376}
  \newrgbcolor{BrickRed}{1.  .3592  .3232}
  \newrgbcolor{Red}{1  0  0}
  \newrgbcolor{OrangeRed}{1  0  .50}
  \newrgbcolor{RubineRed}{1  0  .87}
  \newrgbcolor{WildStrawberry}{1  .04  .61}
  \newrgbcolor{CarnationPink}{1  .37  1}
  \newrgbcolor{Salmon}{1  .47  .62}
  \newrgbcolor{Magenta}{1  0  1}
  \newrgbcolor{VioletRed}{1  .19  1}
  \newrgbcolor{Rhodamine}{1  .18  1}
  \newrgbcolor{Mulberry}{.6668  .1180  1.}
  \newrgbcolor{RedViolet}{.9538  .4060  1.}
  \newrgbcolor{Fuchsia}{.5676  .1628  1.}
  \newrgbcolor{Lavender}{1  .52  1}
  \newrgbcolor{Thistle}{.88  .41  1}
  \newrgbcolor{Orchid}{.68  .36  1}
  \newrgbcolor{DarkOrchid}{.60  .20  .80}
  \newrgbcolor{Purple}{.55  .14  1}
  \newrgbcolor{Plum}{.50  0  1}
  \newrgbcolor{Violet}{.98 .15 .95}
  \newrgbcolor{RoyalPurple}{.25  .10  1}
  \newrgbcolor{BlueViolet}{.84  .38  .98}
  \newrgbcolor{Periwinkle}{.43  .45  1}
  \newrgbcolor{CadetBlue}{.38  .43  .77}
  \newrgbcolor{CornflowerBlue}{.35  .87  1}
  \newrgbcolor{MidnightBlue}{.4414  .9259  1.}
  \newrgbcolor{NavyBlue}{.06  .46  1}
  \newrgbcolor{RoyalBlue}{0  .50  1}
  \newrgbcolor{Blue}{0  0  1}
  \newrgbcolor{Cerulean}{.06  .89  1}
  \newrgbcolor{Cyan}{0  1  1}
  \newrgbcolor{ProcessBlue}{.04  1  1}
  \newrgbcolor{SkyBlue}{.38  1  .88}
  \newrgbcolor{Turquoise}{.15  1  .80}
  \newrgbcolor{TealBlue}{.1572  1.  .6668}
  \newrgbcolor{Aquamarine}{.18  1  .70}
  \newrgbcolor{BlueGreen}{.15  1  .67}
  \newrgbcolor{Emerald}{0  1  .50}
  \newrgbcolor{JungleGreen}{.01  1  .48}
  \newrgbcolor{SeaGreen}{.31  1  .50}
  \newrgbcolor{Green}{0  1  0}
  \newrgbcolor{ForestGreen}{.1992  1.  .2256}
  \newrgbcolor{PineGreen}{.3100  1.  .5575}
  \newrgbcolor{LimeGreen}{.50  1  0}
  \newrgbcolor{YellowGreen}{.56  1  .26}
  \newrgbcolor{SpringGreen}{.74  1  .24}
  \newrgbcolor{OliveGreen}{.6160  1.  .4300}
  \newrgbcolor{RawSienna}{.53  .28  .16}
  \newrgbcolor{Sepia}{1.  .7510  .70}
  \newrgbcolor{Brown}{.41  .25  .18}
  \newrgbcolor{TAN}{.86  .58  .44}
  \newrgbcolor{Gray}{1.  1.  1.}
  \newrgbcolor{Black}{1  1  1}
  \newrgbcolor{White}{1  1  1}
  \theoremstyle{plain}
\newtheorem{theorem}{Theorem}[section]

\newtheorem{lemma}{Lemma}[section]
\newtheorem{corollary}{Corollary}[section]
\theoremstyle{definition}

\newtheorem{remark}{Remark}[section]

\newtheorem{remark-definition}{Remark and Definition}[section]
\newtheorem{rem-not}{Remark and Notation}[section]

\begin{document}

\title{Solutions of Fixed Period in the Nonlinear Wave Equation on Networks}
\author{Carlos Garc\'{\i}a-Azpeitia\thanks{{\small Departamento de Matem\'{a}ticas,
Facultad de Ciencias, Universidad Nacional Aut\'{o}noma de M\'{e}xico, 04510
M\'{e}xico DF, M\'{e}xico, cgazpe@ciencias.unam.mx}}, Wieslaw
Krawcewicz\thanks{{\small Applied Mathematics Center at Guangzhou University, Guangzhou, 510006 China,
 and Department of Mathematical Sciences the University of
Texas at Dallas Richardson, 75080 USA. wieslaw@utallas.edu}} and Yanli
Lv\thanks{{\small Corresponding author, College of Science and Three Gorges Mathematical Research Center, China Three Gorges University, Yichang,
443002, P.R.C., yxl103720@utdallas.edu}}}
\date{}
\maketitle

\begin{abstract}
The wave equation on network is defined by $\partial_{tt}u=\Delta_{G}u+g(u)$,
where $u\in\mathbb{R}^{n}$ and the graph Laplacian $\Delta_{G}$ is an operator
on functions on $n$ vertices. We suppose that $g:\mathbb{R}^{n}\rightarrow
\mathbb{R}^{n}$ is an odd continuous function that satisfies $g(0)=g^{\prime
}(0)=0$ and the Nagumo condition. Assuming that the graph is invariant by a
subgroup of permutations $\Gamma$, using a $\Gamma$-equivariant topological
invariant we prove the existence of multiple non-constant $p$-periodic
solutions characterized by their symmetries.


\end{abstract}

\section{Introduction}

       The analysis and understanding of network structures is prominent in the
characterization of patterns in molecules, neural networks, and electric
networking. The structure for representing a network is a graph, given by a
collection of nodes connected by edges. The wave equation on a network is
\[
\ddot{u}=\Delta_{G}u+g(u),\qquad u=(u_{1},u_{2},\dots,u_{n})\in\mathbb{R}
^{n},\;g(u)=(g(u_{1}),...,g(u_{n})),
\]
where $-\Delta_{G}$ represents the graph Laplacian. The semidefinite negative
matrix $\Delta_{G}$ has eigenvalues $-\omega_{n-1}^{2}\leq...\leq-\omega
_{0}^{2}=0$. The eigenvector $e_{0}=(1,...,1)\in\mathbb{R}^{n}$ corresponding
to the eigenvalue $\omega_{0}=0$ is known as the Goldstone mode.

The wave equation on networks has symmetries that appear naturally and given
by the subgroup of permutations that leave invariant graph. The wave equation on a network for the nonlinearity $g(x)=-x^{3}$ is
used as model of electric networks according to \cite{caputo} and \cite{aoki};
\cite{aoki} analyses the existence of special nonlinear periodic solutions in
network of cycles with symmetries $D_{n}$, while \cite{caputo} extends this
result to general networks with $\mathbb{Z}_{2}$-symmetries. Related models
with $D_{n}$-symmetries are analyzed in \cite{Ga15a} and \cite{GolSchSt} .

Let $\Gamma$ denote the group of permutations that preserve the symmetries of
the graph. Our purpose is to study a general autonomous second order system of
ODEs
\begin{equation}
\ddot{u}=f(u),\quad u\in V:=\mathbb{R}^{n}, \label{shs1}
\end{equation}
where $V$ is an orthogonal $\Gamma$-representation and $f:V\rightarrow V$ is a
$\Gamma$-equivariant odd $C^{1}$-function satisfying the standard Nagumo
condition:
\begin{equation}
\exists_{M>0}\;\forall_{u\in V}\;\;\;\left\vert u\right\vert >M\Longrightarrow
u\bullet f(u)>0. \label{eq:cond*}
\end{equation}
Assuming that $0$ is an isolated stationary solution, we establish existence
and multiplicity of various periodic solutions with a fixed period $p>0$.

This general result applies to wave equation on a network, where hypothesis (
\ref{eq:cond*}) is satisfied, for example, by the nonlinearities $g(x)=x^{3}$
and $g(x)=-x^{3}+x^{5}$. To illustrate our result we consider examples of $n$
-chains with $\mathbb{Z}_{2}$-symmetry, a $n$-cycle with dihedral symmetry $
D_{n}$ and the graph of a truncated octahedron with octahedral symmetries $
S_{4}$. In these cases, we prove the existence of at least one periodic
solutions for each period $p>2\pi /\omega _{n-1}$ such that $p\neq 2\pi
k/\omega _{j}$ for $j=1,...,n-1$, where $\omega _{n-1}^{2}$ is the biggest
eigenvalue of $-\Delta _{G}$. In addition, we prove the existence of other
solutions with specific spatio-temporal symmetries for each period $p>2\pi
/\omega _{n-1}$.

The existence of periodic solutions with fixed period have been obtained for
asymptotically linear Hamiltonian systems in \cite{AmZe}, \cite{Ch}, \cite{Go}. In \cite{AmZe} was obtained the first result by comparing a
topological index at zero and infinity. In \cite{Ch} the results are
obtained using Morse theory, while in \cite{Go}, by the use of $S^{1}$
-gradient degree. The Nagumo condition (\ref{eq:cond*}) is equivalent to the
condition that the Hamiltonian system is asymptotic to the identity at
infinite. Then the existence of at least one periodic solution with period  $
p>2\pi /\omega _{n-1}$ is similar to these results, although they do
consider the symmetries of $\Gamma $ that allow us to obtain more solutions
with specific spatio-temporal symmetries.

It is important to mention that we prove the existence of periodic solutions
in the large for non-gradient nonlinearities $g$, i.e. our results are
conceptually different from the continuations of normal modes from the trivial
solution $u=0\ $obtained for gradient systems in \cite{aoki}, \cite{caputo}, \cite{Ga15a} and references therein. Our general result is also applicable to
other models of interests such as the equations with Kuramoto coupling given
by
\[
\ddot{u}_{j}=\sum_{k}a_{j,k}\sin(u_{j}-u_{k})+g(u_{j}),
\]
where $A=(a_{j,k})$ represents the adjacent matrix for the graph and $g$ is a
nonlinearity that satisfies $g(0)=g^{\prime}(0)=0$ and $\left\vert
g(u)\right\vert >1$ for $\left\vert u\right\vert >M$.

The subgroup of permutations $\Gamma$ that leaves the graph invariant
determines the equivariance of the wave equations. The equivariant degree
method \cite{DKLP} allows to detect non-constant $p$-periodic solutions with
various types of spatio-temporal symmetries. Actually, system \eqref{shs1}
without spacial symmetries $\Gamma$ still leads to an equivariant problem in
functional spaces when looking for periodic solutions with fixed period.
Indeed, since system \eqref{shs1} is time-reversible, the related functional
operator is $O(2)$-equivariant. The additional assumption that $g$ is odd
implies that these functional operators are also ${\mathbb{Z}}_{2}
$-equivariant, where ${\mathbb{Z}}_{2}=\{1,-1\}$ acts by multiplication.
Therefore, the associated functional equation is symmetric under the action of
the group
\[
G=\Gamma\times{\mathbb{Z}}_{2}\times O(2)\text{.}
\]

By applying a $H$-fixed point reduction with $H$ a subgroup of $G$, we are
able to exclude constant solutions of the map reduced to the fixed point space
of $H$. We use the groups $H=\{(1,1),(-1,-1)\}$ (which gives anti-periodic solutions $x(t)=-x(t+\pi)$)  and $H=\{(1,1),(-1,\kappa)\}$
(which gives odd solutions
$x(t)=-x(t)$),  among other groups. This allows us to obtain results even in the
case that $0$ is not a regular point of system \eqref{shs1} such as in the
wave equation on networks due to the existence of the Goldstone mode. Since
the reduced map admits the symmetry group $W(H)$, we associate with the system
\eqref{shs1} an equivariant invariant that compares the equivariant degrees
near the zero solution with the equivariant degree on a large ball (which is
determined using \textit{apriori} bounds). The $W(H)$-equivariant invariant
associated with the system \eqref{shs1} reflects the complete structure of the
set of non-constant $p$-periodic solutions according to their symmetric
properties (orbit types). G.A.P. programming, see \cite{Pin}, allows to make
symbolic computations of the equivariant invariants for several types of
$\Gamma$-symmetric network.

The approach to finding periodic solutions with fixed period using equivariant
degree has similarities with the approaches in
\cite{FRR,FRRuan,GolRyb,RR,RY2,RY4,RY3,RybSurvey}. The equivariant degree
methods were used to study the existence of solutions for BVPs in second order
ODEs (see \cite{BKLN,BLN}) and in the case of a reversible system of FDEs. For
Newtonian systems with or without symmetries, the equivariant degree methods
were applied in \cite{DKLP,FRRuan,FRR,RR,RY1,RY2,RY4,RY3,RybSurvey} and for
general Hamiltonian systems in \cite{GolRyb,Radzki, RaRy}. We should also
mention other related work, see \cite{survey,GI2,GI5}. We refer to
\cite{survey,AED,G,GR} for more details about various equivariant degrees and
their properties.

In section 2 we set the problem of finding periodic solutions in functional
space.  In section 3, we  present the method of reduction to the $H$-fixed point space and formulate the abstract existence result 
Theorem \ref{thm:exist2}.  In Section 4 we present three examples of networks, for which we provide a symmetric classification of periodic solutions.

\section{Second Order Autonomous Systems}

\label{sec:sns} Assume that $p>0$ is an arbitrary number. Let $\Gamma$ be a
finite group and $V=\mathbb{R}^{n}$ an orthogonal representation of $\Gamma$
($\Gamma$ is acting on $\mathbb{R}^{n}$ by permuting the vector coordinates in
$\mathbb{R}^{n}$). We define the following second order autonomous system:
\begin{equation}
\begin{cases}
\ddot{u}(t)=f(u(t)),\quad t\in\mathbb{R},\;u(t)\in V,\\
u(t)=u(t+p),\;\dot{u}(t)=\dot{u}(t+p),
\end{cases}
\label{eq:system1}
\end{equation}
where $f:V\rightarrow V$ is a $C^{1}$-function satisfying the following assumptions:

\begin{itemize}
\item[(A1)] $f$ is $\Gamma$-equivariant, i.e., $f(\gamma u)=\gamma f(u)$ for
all $\gamma\in\Gamma$ and $u\in V$;

\item[(A2)] $f$ is odd function, i.e., $f(-u)=-f(u)$, for all $u\in V$, i.e.
$f(0)=0$;

\item[(A3)] $\exists_{ M>0}~\forall_{ u\in V}\;\;\left\vert u\right\vert
>M\;\Longrightarrow\; u\bullet f(u)>0$
\end{itemize}

By substituting $x(t)=u(\lambda t)$, the system \eqref{eq:system1} is
transformed to
\begin{equation}
\ddot{x}(t)=\lambda^{2}f(x(t)),\qquad\text{where }\lambda=p/2\pi\text{.}
\label{eq:yy0}
\end{equation}
Notice that $p$-periodic solution $u(t)$ to system \eqref{eq:system1}
corresponds to $2\pi$-periodic solution to system \eqref{eq:yy0}. Therefore,
we reduce the system \eqref{eq:system1} to
\begin{equation}
\begin{cases}
\ddot{x}(t)=\lambda^{2}f(x(t)),\quad t\in\mathbb{R},\;x(t)\in V,\\
x(t)=x(t+2\pi),\;\dot{x}(t)=\dot{x}(t+2\pi),
\end{cases}
\label{eq:system2}
\end{equation}
Clearly, the function $f$ is a $C^{1}$-function and satisfies conditions (A1)--(A3).

We do not assume that $0$ is non-degenerate equilibrium, i.e. the spectrum
$\sigma(A)$ of the matrix $A:=\lambda^{2}Df(0)$ may contain $-k^{2}$ for some
$k=0,1,2,\dots$.

\subsection{Sobolev Spaces of $\mathbf{2\pi}$-Periodic Functions}

\label{sec:H} Let $\mathscr H$ denote the second Sobolev space of $2\pi
$-periodic functions from $\mathbb{R}$ to $V$, i.e.,
\[
\mathscr H:=H_{2\pi}^{2}(\mathbb{R},V)=\left\{  x:\mathbb{R}\rightarrow
V:x(0)=x(2\pi),\;x|_{[0,2\pi]}\in H^{2}([0,2\pi];V)\right\}  ,
\]
equipped with the inner product
\[
\langle x,y\rangle:=\int_{0}^{2\pi}\left(  \ddot{x}(t)\bullet\ddot{y}
(t)+\dot{x}(t)\bullet\dot{y}(t)+x(t)\bullet y(t)\right)  dt,
\]
and the associated norm $\left\Vert x\right\Vert _{\mathscr H}~.$

Let $O(2)$ denote the group of $2\times2$-orthogonal matrices. Notice that
$O(2)=SO(2)\cup SO(2)\kappa$, where
\[
\begin{bmatrix}
\cos\tau & -\sin\tau\\
\sin\tau & \cos\tau
\end{bmatrix}
\in SO(2),\qquad\kappa=
\begin{bmatrix}
1 & 0\\
0 & -1
\end{bmatrix}
.
\]
It is convenient to identify a rotation with $e^{i\tau}\in S^{1}
\subset\mathbb{C}$. Notice that $\kappa e^{i\tau}=e^{-i\tau}\kappa$.

Put $G=\Gamma\times{\mathbb{Z}}_{2}\times O(2)$, then the space $\mathscr H$
is an orthogonal Hilbert representation of $G$. Indeed, for $x\in\mathscr H$,
$\gamma\in\Gamma$ and $e^{i\tau}\in S^{1}$, we can define the group action as
\begin{align*}
(\gamma,\pm1,e^{i\tau})x(t)  &  =\pm\gamma x(t+\tau),\\
(\gamma,\pm1,e^{i\tau}\kappa)x(t)  &  =\pm\gamma x(-t+\tau),
\end{align*}
and $\Gamma$ acting on $V=\mathbb{R}^{n}$ by permuting the vector coordinates.

In a standard way we identify a $2\pi$-periodic function $x:\mathbb{R}
\rightarrow V$ with a function $x:S^{1}\rightarrow V$, so we can write
$H^{2}(S^{1},V)$ instead of $H^{2}(\mathbb{R},V)$. Consider the $O(2)$-isotypical decomposition of $\mathscr H~,$
\begin{equation}
\mathscr H=\overline{\bigoplus_{k=0}^{\infty}\mathbb{V}_{k}},\quad
\mathbb{V}_{k}:=\left\{  u_{k}\cos(kt)+v_{k}\sin(kt):u_{k},v_{k}\in V\right\}
. \label{eq:isoS1}
\end{equation}

Notice that for $k>0$ the $SO(2)$-space $\mathbb{V}_{k}$ can be identified
with the complexification $V^{c}:=V\oplus iV$ of $V$, on which $SO(2)$ acts
by
\[
e^{i\theta}(a+ib)=e^{-ik\theta}\cdot(a+ib),\quad a,\,b\in V,
\]
where `$\cdot$' stands for complex multiplication. Indeed, define the real
isomorphism $\psi_{k}:V^{c}\rightarrow\mathbb{V}_{k}$ by $\psi(a+ib)(t)=\cos
(kt)a+\sin(kt)b$, $a$, $b\in V$. Then
\[
\psi_{k}\big(e^{i\theta}(a+ib)\big)=\cos(k(t+\theta))a+\sin(k(t+\theta
))b=e^{i\theta}\psi_{k}(a+ib).
\]
On the other hand, the operator $\kappa$ acts on the space $\mathbb{V}_{k}$ by
complex conjugation, i.e.
\[
\kappa\psi(a+ib)(t)=\cos(kt)a-\sin(kt)b=\psi(\kappa(a+ib)).
\]

In summary, the $O(2)$-isotypical component $\mathbb{V}_{0}$ represents
constant functions which can be identified with $V$, while $\mathbb{V}_{k}$,
for $k=1,2,\dots$, is equivalent to the complexification $V^{c}=V\oplus iV$
where $SO(2)$ acts by $k$-folded rotations, i.e., $\gamma z:=\gamma^{k}\cdot
z$, $\gamma\in S^{1}\simeq SO(2)$, $z=x+iy=(x,y)\in\mathbb{R}^{2}$ and
`$\cdot$' denotes the complex multiplication, and $\kappa$ acts by complex
conjugation. That means $\mathbb{V}_{k}$ is modeled on the irreducible
$O(2)$-representation $\mathcal{V}_{k}\simeq\mathbb{R}^{2}$ with $k$-folded
action of $SO(2)$. Then \eqref{eq:isoS1} is also ${\mathbb{Z}}_{2}\times
O(2)$-isotypical decomposition of $\mathscr H$, where ${\mathbb{Z}}_{2}$ acts
by multiplication.

\subsection{Setting in Functional Spaces}

Define the operator:
\[
L:\mathscr H\longrightarrow L^{2}(S^{1};\mathbb{V}),\quad Lx:=\ddot{x}-x,
\]
Then the system $\eqref{eq:system2}$ is equivalent to
\begin{equation}
Lx=\lambda^{2}f(x)-x,\;x\in\mathscr H~. \label{yy8}
\end{equation}
Since $L$ is an isomorphism, equation \eqref{yy8} can be reformulated as
\begin{equation}
\mathscr F(x):=x-L^{-1}(\lambda^{2}f(x)-x)=0~. \label{eq:bigF}
\end{equation}

Notice that $x\equiv0$ is a solution to the equation $\mathscr F(x)=0$. Put
\[
\mathscr A:=D\mathscr F(0)=\id-L^{-1}(\lambda^{2}Df(0)x-x):\mathscr
H\longrightarrow\mathscr H.
\]
One can easily check that $\mathscr A$ is a Fredholm operator of index zero.
Therefore, $\mathscr A$ is an isomorphism if and only if $0\notin
\sigma(\mathscr A)$. Since $\mathscr A$ is $O(2)$-equivariant, it preserves
the isotypical decomposition \eqref{eq:isoS1}. The operator $L$ on the
isotypical components $\mathbb{V}_{k}$ is given by,
\[
L(\cos(kt)u_{k}+\sin(kt)v_{k})=-(k^{2}+1)(\cos(kt)u_{k}+\sin(kt)v_{k}),
\]
which implies
\[
-L^{-1}(u_{k}\cos(kt)+v_{k}\sin(kt))=\frac{1}{k^{2}+1}(u_{k}\cos(kt)+v_{k}
\sin(kt)).
\]
Therefore, in $\mathbb{V}_{k}$ we have
\[
\mathscr A|_{\mathbb{V}_{k}}=\text{Id}-\frac{1}{1+k^{2}}(\lambda
^{2}Df(0)-\text{Id}),
\]
and we obtain the following description of the spectrum of $\mathscr A$
\begin{equation}
\sigma(\mathscr A)=\left\{  \frac{k^{2}+\lambda^{2}\mu}{1+k^{2}}:\mu\in
\sigma(Df(0)),\quad k\in\mathbb{N}\right\}  . \label{eq:spA}
\end{equation}

\subsection{Apriori Bounds}

Consider the following homotopy of the operator $\mathscr F$ (cf.
\eqref{yy8})
\begin{equation}
\mathscr F_{\tau}(x):=x-\tau L^{-1}(\lambda^{2}f(x)-x),\quad\tau\in
\lbrack0,1].\;
\end{equation}
Notice that $\mathscr F:=\mathscr F_{1}$ and $\mathscr F_{\tau}$ is a
$G$-equivariant completely continuous field. Then we have

\begin{lemma}
\label{lem.1} Assume that $f:V\rightarrow V$ is a continuous function
satisfying conditions (A1)--(A3) and (A4). If $x(t)$ is a $2\pi$-periodic
function of class $C^{2}$ such that $\max_{t\in\mathbb{R}}\Vert x(t)\Vert>M$,
then $x(t)$ is not a solution of $\mathscr F_{\tau}(x)=0$ for $\tau\in
\lbrack0,1]$.
\end{lemma}

\begin{proof}
Assume the contradiction that $x(t)$ is a solution of $\mathscr F_{\tau}(x)=0$
when $\max_{t\in\mathbb{R}}\Vert x(t)\Vert>M$. Consider the function
$\phi(t):=\frac{1}{2}\Vert x(t)\Vert^{2}$. Suppose that $\phi(t_{0}
)=\max_{t\in\mathbb{R}}\phi(t)$, then $\phi^{\prime}(t_{0})=x(t_{0}
)\bullet\dot{x}(t_{0})=0$ and $\phi^{\prime\prime}(t_{0})=\dot{x}
(t_{0})\bullet\dot{x}(t_{0})+\ddot{x}(t_{0})\bullet x(t_{0})\leq0$. However,
by condition (A3),
\[
\phi^{\prime\prime}(t_{0})=\dot{x}(t_{0})\bullet\dot{x}(t_{0})+\ddot{x}
(t_{0})\bullet x(t_{0})>(1-\tau)\left\vert x(t_{0})\right\vert ^{2}
+\tau\lambda^{2}f(x(t_{0}))\bullet x(t_{0})>0,
\]
which is a contradiction.
\end{proof}

\begin{lemma}
\label{lem1} There exists $R>0$ such that every solution $x(t)$ of $\mathscr
F_{\tau}(x)=0$ for $\tau\in\lbrack0,1]$ satisfies $\left\Vert x\right\Vert
_{\mathscr H}<R$.
\end{lemma}

\begin{proof}
By lemma $\ref{lem.1}$, there exists a $M>0$ such that any $2\pi$-periodic
solution $x(t)$ to $\mathscr F_{\tau}(x)=0$ satisfies $\left\vert
x(t)\right\vert <M$. By Sobolev embedding $\mathscr H=H_{2\pi}^{2}
(\mathbb{R},V)$ is contained in $C_{2\pi}^{0}(\mathbb{R},V)$. Then there is a
constant $C>0$ such that
\[
\left\Vert x\right\Vert _{\mathscr H}\leq C\left\Vert x\right\Vert _{C_{2\pi
}^{0}}\leq MC.
\]
Take $R=CM$.
\end{proof}

By Lemma $\ref{lem1}$, there exists a constant $R>0$ such that any solution to
$\mathscr F_{\tau}=0$ must belong to the set $\Omega_{R}:=\left\{
x\in\mathscr H:\Vert x\Vert<R\right\}  $. In particular, $\left\{  \mathscr
F_{\tau}\right\}  _{\tau\in\lbrack0,1]}$ is an $\Omega_{R}$-admissible
homotopy between $\mathscr F_{1}=\mathscr F$ and $\mathscr F_{0}=\id$.

\section{Reduction to Fixed-Point Subspace of $H$}

We consider the system \eqref{eq:system2} where $f:V\rightarrow\mathbb{R}$ is
differentiable at $0$ satisfying the assumptions (A1)-(A3). We do not exclude
the case that $0\in\sigma(\mathscr A)$ (which means that zero is
\textit{degenerate} solution to \eqref{eq:system2}). Notice that
$\text{\textrm{Ker\thinspace}}(\mathscr A)$ is finite dimensional because
$\mathscr A:\mathscr H\rightarrow\mathscr H$ is a Fredholm operator.

Since for any closed subgroup $H\subset G$, the nonlinear map $\mathscr
F^{H}:=\mathscr F|_{\mathscr H^{H}}:\mathscr H^{H}\rightarrow\mathscr H^{H}$
and $\mathscr A^{H}:=\mathscr A|_{\mathscr H^{H}}:\mathscr H^{H}
\rightarrow\mathscr H^{H}$ are $W(H)$-equivariant, it is a common practice to
look for the $W(H)$-orbits of solutions for \eqref{eq:system2} in the subspace
$\mathscr H^{H}$. In other words, if for some $x\in\mathscr H^{H}$ we have
$\mathscr F^{H}(x)=0$, then $\mathscr F(x)=0$, i.e. $x$ is a solution to
\eqref{eq:system2}. Let us find groups $H\leq\mathbb{Z}_{2}\times O(2)$ such
that
\[
\text{\textrm{Ker \thinspace}}(A)\cap\mathscr H^{H}=\{0\}.
\]

Let $m$ be a positive integer. Notice that $-1\in\mathbb{Z}_{2}$ acts always
as $-\id$ on the representations $\mathbb{V}_{m}$. Since $e^{i\pi/m}\in O(2)$
acts as $-\id$ on the representations $\mathbb{V}_{m}$, then the element
$(-1,e^{i\pi/m})\in\mathbb{Z}_{2}\times O(2)$ acts trivially on each
representation $\mathbb{V}_{m}$, i.e. each isotropy group of the
representation $\mathbb{V}_{m}$ contains the group generated by $(-1,e^{i\pi
/m})$:
\begin{equation}
H={\mathbb{Z}}_{2m}^{d}:=\left\{  e\right\}  \times\{(1,1),(-1,\gamma
),(1,\gamma^{2}),\dots,(-1,\gamma^{2m-1})\},\qquad\gamma=e^{i\pi/m}\text{.}
\label{H}
\end{equation}

The action of the generator $(-1,\gamma)$ on the function $\cos(kt)u_{k}
+\sin(kt)v_{k}\in\mathscr H^{{\mathbb{Z}}_{2m}^{d}}$ is given by
\[
(-1,\gamma)(\cos(kt)u_{k}+\sin(kt)v_{k})=-\left(  \cos\left(  kt+\frac{k\pi
}{m}\right)  u_{k}+\sin\left(  kt+\frac{k\pi}{m}\right)  v_{k}\right)  .
\]
Therefore, these functions are in the fixed point space $\mathscr H^{H}$ if
$k$ is an odd multiple of $m$, and
\[
\mathscr H^{H}=\overline{\bigoplus_{k\in m(2\mathbb{N}-1)}\mathbb{V}_{k}}
\]

Since the fix point space of the group $H$ do not contain constant functions,
then it is enough to consider the fix point space,
\[
\mathscr H^{H}=\{x\in\mathscr H:x_{j}(t)=-x_{j}(t+\pi/m)=x_{j}(t+2\pi/m)\},
\]
to exclude the zero eigenvalues of $\mathscr A$. Notice that $H$ is normal,
thus $N(H)=\Gamma\times{\mathbb{Z}}_{2}\times O(2)$ and
\[
W(H)=\Gamma\times{\mathbb{Z}}_{2}\times O(2).
\]
Therefore, we have that $\mathscr H^{H}$ is a Hilbert representation of
$W(H)=G$ and $\mathscr F^{H}$ is $G$-equivariant. Then we can look for
non-constant periodic solutions of the equation $\mathscr F^{H}(x)=0$.

Since each of the subspaces $\mathbb{V}_{k}^{H}$ is $\Gamma$-invariant, we
obtain the $G$-invariant decomposition of the space $\mathscr H^{H}$. Then for
all considered above twisted subgroups $H$ the component $\mathbb{V}_{k}^{H}$
can be refined into the direct product of $G$-isotypical components
\begin{equation}
\mathbb{V}_{k}^{H}=V_{1,k}^{H}\oplus V_{2,k}^{H}\oplus\dots\oplus V_{r,k}^{H},
\label{eq:iso-red}
\end{equation}
where $V_{j,k}^{H}$ is modeled on the irreducible $\Gamma$-representation
$\mathcal{V}_{j,k}$. \vskip.3cm

Put $\Omega_{\varepsilon}=\{x\in\mathscr H:\Vert x\Vert<\varepsilon\}$. Since
$0\notin\sigma(\mathscr A^{H})$, one can easily show  that for
sufficiently small $\varepsilon>0$, the map $\mathscr F^{H}$ is $\Omega
_{\varepsilon}^{H}$-admissibly $G$-equivariantly homotopic to $\mathscr A^{H}
$. On the other hand, by Lemma \ref{lem1}, the map is also $\Omega_{R}^{H}
$-admissibly $G$-equivariantly homotopic to $\id$. Using the notation
introduced in Appendix the equivariant topological invariant
\begin{equation}
\omega^{H}=G\mbox{\rm -}\deg(\mathscr F^{H},\Omega^{H}), \label{eq:inv-H}
\end{equation}
where $\Omega:=\Omega_{R}\setminus\Omega_{\varepsilon}$ is well-defined and is
given by
\[
\omega^{H}=G\mbox{\rm -}\deg(\mathscr F^{H},\Omega_{R}^{H})-G\mbox{\rm -}\deg
(\mathscr F^{H},\Omega_{\varepsilon}^{H})=(G)-G\mbox{\rm -}\deg(\mathscr
A,B(\mathscr H))\text{.}
\]

On the other hand, we have
\begin{equation}
G\mbox{\rm -}\deg\left(  \mathscr A,B(\mathscr H)\right)  =\prod_{k}
\prod_{j=1}^{r}\prod_{\mu\in\sigma(A)}\left(  \deg_{\mathcal{V}_{j,k}}\right)
^{\mathfrak{m}_{-}^{j,k}(\mu)}, \label{eq:prod-2}
\end{equation}
where $\mathfrak{m}_{-}^{j,k}(\mu)$ stands for the $\mathcal{V}_{j,k}
$-isotypical multiplicity of the eigenvalue
\[
\xi_{j,k}=\frac{k^{2}+\lambda^{2}\mu_{j}}{k^{2}+1}.
\]
with $\mu_{j}\in\sigma(Df(0))$\vskip.3cm

Consequently we can formulate the following theorem for the resonance case. \vskip.3cm

\begin{theorem}
\label{thm:exist2} Let $f:V\rightarrow V$ satisfies the conditions (A1), (A2),
(A3). Set $H={\mathbb{Z}}_{2m}^{d}$ and assume that
\[
\lambda^{2}\neq-k^{2}/\mu,\qquad k\in m(2\mathbb{N}+1),\mu\in\sigma(Df(0)).
\]
If the equivariant invariant $\omega^{H}$ is not equal to zero, then the
system \eqref{eq:system2} admits a non-constant $2\pi$-periodic solution. More
precisely, if
\[
\omega^{H}=n_{1}(H_{1})+n_{2}(H_{2})+\dots+n_{k}(H_{k}),\quad n_{j}
\not =0,\;\;j=1,2,\dots,k,
\]
then for every $(H_{j})$, there exists a non-constant $2\pi$-periodic solution
$x(t)$ to \eqref{eq:system2} such that $G_{x}\geq H_{j}$.
\end{theorem}

\begin{proof}
This result is a direct consequence of the existence property (G1) for the
equivariant $G$-degree.
\end{proof}

\begin{remark}
We may consider other fixed point groups $H$ in this theorem. The subgroup
$D_{m}^{z}\subset\Gamma\times{\mathbb{Z}}_{2}\times O(2)$ (see \cite{AED} for
more details related to the conventions used in this paper), where
\begin{equation}
D_{m}^{z}:=\{e\}\times\{(1,1),(1,\gamma),\dots,(1,\gamma^{m-1}),(-1,\kappa
),(-1,\gamma\kappa),\dots,(-1,\gamma^{m-1}\kappa)\},
\end{equation}
where $\gamma=e^{i\frac{2\pi}{m}}$. Clearly, a function $x\in\mathscr H$
belongs to $\mathscr H^{D_{m}^{z}}$ if and only if $x$ is a $2\pi/m$-periodic
odd function. Then
\[
\mathbb{V}_{k}^{D_{m}^{z}}:=
\begin{cases}
\left\{  0\right\}  & \mbox{if}\;k\;\text{is not a multiple of}\;m,\\
\left\{  \sin(kt)u:u\in V\right\}  & \mbox{if}\;k\;\text{is a multiple of}\;m.
\end{cases}
\]
Notice that $N(D_{m}^{z})=\Gamma\times{\mathbb{Z}}_{2}\times D_{2m}$, then
$W(D_{m}^{z})=\Gamma\times V_{4}$, where $V_{4}:={\mathbb{Z}}_{2}
\times{\mathbb{Z}}_{2}$ is the Klein's group.

The subgroup $D_{2m}^{d}\subset\Gamma\times{\mathbb{Z}}_{2}\times O(2)$,
where
\begin{equation}
D_{2m}^{d}:=\{e\}\times\{(1,1),(-1,\gamma),\dots,(-1,\gamma^{2m-1}
),(-1,\kappa),(1,\gamma\kappa),\dots,(1,\gamma^{2m-1}\kappa)\}.
\end{equation}
Then, clearly
\[
\mathbb{V}_{k}^{D_{2m}^{d}}:=
\begin{cases}
\left\{  \sin(kt)u:u\in V\right\}  & \mbox{if}\;k=m(2r+1),\text{for some
}\;r=0,1,2,\dots\\
\left\{  0\right\}  & \text{otherwise}.
\end{cases}
\]
Notice that $N(D_{2m}^{d})=\Gamma\times{\mathbb{Z}}_{2}\times D_{2m}$, thus
$W(D_{2m}^{d})=\Gamma\times{\mathbb{Z}}_{2}$.
\end{remark}

\section{Nonlinear wave equations~on networks}

In this section we compute the invariant $\omega^{H}$ and apply the main
theorem to the nonlinear wave equation on networks, given by the function
\[
f(u)=\Delta_{G}u+g(u),\quad u=(u_{1},u_{2},\dots,u_{n})\in\mathbb{R}^{n}.
\]
For simplicity, we assume that $g^{\prime}(0)=0$ and $\left\vert
g(u)\right\vert >\left\vert u\right\vert $ for $\left\vert u\right\vert >M$.
These conditions are satisfied, for example, by the nonlinearities $g(x)=x^{3}$
and $g(x)=-x^{3}+x^{5}$.\ By assumptions $f$ satisfies the Nagumo condition
and
\[
Df(0)=\Delta_{G}u\text{.}
\]

Let $S_{n}$ be the group of permutations of the elements $\{1,...,n\}$. The
equivariance of the operator $f(u)$ is satisfied by the subgroup $\Gamma$ of
$S_{n}$ consisting of those elements $\Gamma<S_{n}$, with action of $\gamma\in
S_{n}$
\[
\gamma u=(u_{\gamma(1)},...,u_{\gamma(n)}), \quad u\in\mathbb{R}^{n},
\]
for which $f(\gamma u)=\gamma f(u)$. For the nonlinearity of the form
$g(u)=(g(u_{1}),...,g(u_{n}))$ the subgroup $\Gamma$ can be determined by
finding the elements $\gamma\in S_{n}$ that commutes with $\Delta_{G}$, i.e.
\begin{equation}
\Delta_{G}\gamma=\gamma\Delta_{G}\text{.}
\end{equation}

The $\Gamma\times{\mathbb{Z}}_{2}$-isotypical decomposition of $V$ is
\[
V=V_{0}^{-}\oplus V_{1}^{-}\oplus\dots\oplus V_{q}^{-},
\]
where each of the subspaces $V_{j}^{-}$ is equivalent to $\mathcal{V}_{j}^{-}$
($\mathcal{V}_{j}^{-}$ is $j$-th irreducible $\Gamma$-representation with
antipodal ${\mathbb{Z}}_{2}$-action). Each eigenvalue $\mu_{j}$ of
$Df(0)=\triangle_{G}$ is assigned to an irreducible representation
$\mathcal{V}_{n_{j}}^{-}$ for $j=1,...,r$. Therefore, we have the following
$G$-invariant decomposition
\[
\mathscr H^{H}=\overline{\bigoplus_{k\in2\mathbb{N}-1,j=0,...,r}
\mathcal{V}_{n_{j},k}}~,
\]
where $\mathcal{V}_{n_{j},k}$ is equivalent to the $k$-fold irreducible
representation $\mathbb{C}\otimes\mathcal{V}_{n_{j}}$ where $\triangle_{G}$
has eigenvalue $\mu_{j}$.

We order the eigenvalues $\mu_{j}$ for $j=0,..,r$ of $\Delta_{G}$ by
\[
\mu_{r}\leq...\leq\mu_{0}=0~.
\]
Since $\mu_{0}=0$, based on our earlier discussion, it is enough to consider
(for the fixed-point reduction) the group $H=\mathbb{Z}_{2}^{-}$ (here
$\bz_{2}^{-}=\bz_{2}^{d}$). For $j=1,$ we have that $\xi_{1,k}>0$ for all $k$.
For $j\neq0$, the eigenvalues $\xi_{j,k}$ changes sign from positive to
negative when $\lambda$ crosses $k/\omega_{j}$. We define an ordering of the
values $k/\omega_{j}$ for $k\in2\mathbb{N}-1$ and $j=1,...,r$ as
\[
1/\omega_{r}=k_{1}/\omega_{j_{1}}\leq k_{2}/\omega_{j_{2}}\leq...\leq
k_{l}/\omega_{j_{l}}~.
\]

Therefore, for any interval $k_{l}/\omega_{j_{l}}<\lambda<k_{l+1}
/\omega_{j_{l+1}}$ we conclude that
\begin{equation}\label{eq:basic-product}
G\text{-deg}(\mathscr A,B(\mathscr H^{H}))=\prod_{k/\omega_{j}<\lambda}
\deg_{\mathcal{V}_{n_{j},k}^{-}}.
\end{equation}
\vs
\begin{remark}\rm Consider an element $a\in A(G)$  given as a (finite) sum
\[
a=\sum_{(H)\in \Phi_0(G)} n_H \, (H)
\]
Then, for every $(H)\in \Phi(G)$  we can define the integer $\langle a,(H)\rangle :=n_H$. We also put 
$\Phi(a):=\{(H): \langle a, (H)\rangle \not=0\}$. Notice that, if $a$ is an invertible element in $A(G)$ then $(G)\in \Phi(a)$. Indeed, since 
$a\cdot a^{-1}=(G)$, by definition of the multiplication in the Burnside ring we have that $\langle (H)\cdot (K),(G)\rangle=(G)$ if and only if $H=K=G$.
Therefore, any invertible element $a\in A(G)$ must be of type $a=\pm (G)+x$, where $(G)\not\in \Phi(x)$. One can easily construct examples of non-invertible elements of type $(G)+x$ in any Burnside ring $A(G)$ (for a non-trivial group $G$).
Notice that for any  orthogonal  $G$-representation $V$ the element $a:=G\text{-deg}(-\id,B(V))$ is invertible (since by the multiplicativity property $a^2=(G)$), therefore all basic $G$-degrees $\deg_{\mathcal{V}_i}$  are invertible in $A(G)$. Finally, let us point out that the Burnside ring $A(G)$ is not an integral domain in general.
\end{remark}
\vs
Clearly, by using the property that $\left(\deg_{\mathcal{V}_{n_{j},k}^{-}}\right)^2=(G)$ in $A(G)$, formula
\eqref{eq:basic-product} can be reduced in such a way that
\begin{equation*}
\prod_{k/\omega_{j}<\lambda}\deg_{\mathcal{V}_{j,k}^{-}}=\prod_{i,l}
\deg_{\mathcal{V}_{i,l}^-},
\end{equation*}
where $\deg_{\mathcal{V}_{i,l}^-}\not=\deg_{\mathcal{V}_{i',l^{\prime}}^-}$ for
$(i,l)\not =(i',l^{\prime})$. In other words, every basic $G$-degree appears only once
in this product. We should also point out that two basic degrees for two
different $G$-representations can be the same (see \cite{AED} for examples). 

\begin{lemma} 
\label{lem:max} Suppose $a$, $b\in A(G)$ and $(H)\in \Phi_0(G)$ be such that 
\begin{itemize}
\item[(i)] $a $ is invertible;
\item[(ii)] $(H)$ is maximal in $\Phi(a)\cup\Phi(b)\setminus \{(G)\}$;
\item[(iii)] $(H)\in \Phi(b)$, $(H)\notin \Phi(a)$;
\end{itemize}
Then $(H)\in \Phi(ab)$.
\end{lemma}

\begin{proof}
Since the element $a$ is invertible, we have $a=\pm (G)+x$ and $b:=n(G)+m(H)+y$, where $m\not=0$. Then we have 
\[
ab=(\pm (G)+x)(n(G)+m(H)+y)=\pm n(G)\pm m(H)+n x+xy.
\]
Since $(L)\in \Phi((K)(K'))$ implies that $L=gKg^{-1}\cap K'$ for some $g\in G$, it follows that $(H)\notin \Phi(x)\cup\Phi(xy)$, thus $\langle ab,(H)\rangle=\pm m\not=0$.
\end{proof}
\begin{corollary} \label{cor:prod} Consider the element 
\[
\omega:=\prod_{i}\deg_{\cV_i},\quad \deg_{\cV_i}\not=\deg_{\cV_{i'}}\;\;\text{ for }\;\; i\not=i'.
\]
Denote by $\Psi(\omega)$ the set of all maximal elements $(H)$ in 
$\bigcup_{i} \Phi(\deg_{\cV_i})\setminus \{(G)\}$ such that if $(H)\in\Phi(\deg_{\cV_i})$ then $(H)\notin\Phi(\deg_{\cV_{i'}})$ for $i'\not=i$. Then for every $(H)\in \Psi(\omega)$, $(H)\in \Phi(\omega-(G))$.

\end{corollary}

\subsection{Chain of $n$ elements with symmetries $\Gamma=\mathbb{Z}_{2}$}

\begin{figure}[ptb]
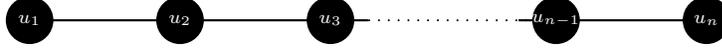

\begin{center}
\vskip2cm\hskip-9cm 
\psline(0,1)(4.5,1)
\psline[linestyle=dotted](4.5,1)(7,1)
\psline(6.5,1)(9,1)
\psdots[linewidth=8pt](0,1)(9,1)(2,1)(4,1)(7,1)
\rput(0,1){\scriptsize \white $u_1$}
\rput(2,1){\scriptsize \white $u_2$}
\rput(4,1){\scriptsize \white $u_3$}
\rput(7,1){\scriptsize\white $u_{n-1}$}
\rput(9,1){\scriptsize\white $u_{n}$}

\end{center}
\caption{Chain of $n$ elements.}
\label{fig:chain}
\end{figure}

For the chain (see Figure \ref{fig:chain}) the matrix $\Delta_{G}$ is given by
\[
\Delta_{G}=\left(
\begin{array}
[c]{ccccccc}
-1 & 1 & 0 &  & 0 & 0 & 0\\
1 & -2 & 1 &  & 0 & 0 & 0\\
&  &  &  &  &  & `\\
0 & 0 & 0 &  & 1 & -2 & 1\\
0 & 0 & 0 &  & 0 & 1 & -1
\end{array}
\right)  .
\]
In this case the matrix $\Delta_{G}$ commutes with $\Gamma=\mathbb{Z}_{2}$,
where the generator $\xi\in\bz_{2}$ acts on $j\in\{1,2,\dots,n\}$,
$\xi(j)=n+1-j\ $(mod $n$). Then for $j=0,1,\dots,n-1$ we have the eigenvalues
$\mu_{j}:=-\omega_{j}^{2}$ of $\Delta_{G}$ (see \cite{caputo}), where
\[
\omega_{j}=2\sin\frac{j\pi}{2n},\;\;\text{ and }\;\;0=\omega_{0}<\omega
_{1}<\dots<\omega_{n-1}.
\]
The corresponding eigenvector to $\mu_{j}$ is given by
\[
v_{j}:=(v_{j}^{1},v_{j}^{2},\dots,v_{j}^{n}),\;\;\;v_{j}^{k}=\sqrt{\frac{2}
{n}}\cos\left(  \frac{j\pi}{n}(k-\textstyle{\frac{1}{2}})\right)
,\;\;k=1,2,\dots,n.
\]
The eigenspace of $\Delta_{G}$ corresponding to $\mu_{j}$ is $E(\mu
_{j})=\text{span}(v_{j})$.

For the above action of $\bz_{2}$ on $V:=\br^{n}$, $\xi$ acts on the
eigenspace $E(\mu_{j})$ by multiplication by $-1$ for $j$ odd and trivially
for $j$ even. Therefore, for $G:=\bz_{2}\times\bz_{2}\times O(2)$, we have the
following $G$-invariant decomposition of $V$
\[
V=\bigoplus_{\text{$j$ even}}E(\mu_{j})\oplus\bigoplus_{\text{$j$ odd}}
E(\mu_{j})=:V_{0}^{-}\oplus V_{1}^{-}.
\]
Let $\mathscr E(\xi_{j,k})$ be the eigenspace corresponding to the eigenvalue
$\xi_{j,k}:=\frac{k^{2}+\lambda^{2}\mu_{j}}{k^{2}+1}$ of the operator
$\mathscr A$ defined in $\mathscr H$. The eigenspace $\mathscr E(\xi_{j,k})$
is equivalent to the irreducible $G$-representation $\mathcal{V}_{0,k}^{-}$,
for $j$ even or $\mathcal{V}_{1,k}^{-}$ for $j$ odd.

Then, we have the following basic degrees
\begin{align*}
\deg_{\mathcal{V}_{0,k}^{-}}  &  =(G)-(\bz_{2}\times D_{2k}^{d})\\
\deg_{\mathcal{V}_{1,k}^{-}}  &  =(G)-(V_{4}^{\bz_{2}^{-}}\times_{\bz_{2}
}D_{2k}).
\end{align*}
where $D_{2k}^{d}\leq\bz_{2}\times O(2)$ is given
\[
D_{2k}^{d}:=\{((-1)^{l},\zeta^{l}),((-1)^{l},\zeta^{l}\kappa):l=0,1,\dots
,2k-1\},\quad\zeta:=e^{\frac{i\pi}{k}}.
\]
The subgroup $V_{4}^{\bz_{2}^{-}}\times_{\bz_{2}}D_{2k}\leq V_{4}\times O(2)$
($V_{4}:=\bz_{2}\times\bz_{2}$) is given by
\[
V_{4}^{\bz_{2}^{-}}\times_{\bz_{2}}D_{2k}=\{(g,h)\in V_{4}\times
O(2):\varphi(g)=\psi(h)\},
\]
where $\varphi:V_{4}\rightarrow\bz_{2}$ is a homomorphism such that
$\ker(\vp)=\{(1,1),(-1,-1)\}=:\bz_{2}^{-}$ and $\psi:D_{2k}\rightarrow\bz_{2}$
is a homomorphism with $\ker(\psi)=D_{k}$.

In general, most of normal modes are non resonant, i.e. for $j\not =k$ we have
$\omega_{j}\neq l\omega_{k}$ for all $l\in\mathbb{N}$. However, there are some
special cases where resonances of this kind exist. In order to show our
results in a non-resonant case, we chose the case $n=4$. In this case we have
\[
\omega_{0}=0<\omega_{1}=\sqrt{2-\sqrt{2}}<\omega_{2}=\sqrt{2}<\omega_{3}
=\sqrt{2+\sqrt{2}},
\]
and
\[
\frac{1}{\omega_{3}}<\frac{1}{\omega_{2}}<\frac{1}{\omega_{1}}<\frac{3}
{\omega_{3}}<\frac{3}{\omega_{2}}<\frac{3}{\omega_{1}}<\dots<\frac
{2k+1}{\omega_{3}}<\frac{2k+1}{\omega_{2}}<\frac{2k+1}{\omega_{1}}<\dots
\]
Therefore, we have the following result: \vs

\begin{theorem}
\label{th:chain} Consider the network consisting of a chain of four elements
and assume that $g^{\prime}(0)=0$ and $\left\vert g(u)\right\vert >\left\vert
u\right\vert $ for $\left\vert u\right\vert >M$. Suppose $\lambda>\frac
{1}{\omega_{3}}$  such that for all $k\in2\mathbb{N}-1$ we have
$\lambda\not =\frac{k}{\omega_{j}}$, $j=1,2,3$. Then the system
\eqref{eq:system1} admits a $p$-periodic solution $u(t)$. Moreover,

\begin{itemize}
\item[(a)] If $\frac{2k+1}{\omega_{1}}<\lambda< \frac{2k+3}{\omega_{3}}$ then
there exists a $p$ periodic solution $u(t)$ such that the solution
$x(t)=u(\lambda t)$ to \eqref{eq:system2} has the isotropy group
$\bz_{2}\times D_{4m+2}^{d}$ for some $m$.

\item[(b)] If 
$\frac{2k+1}{\omega_{2}}<\lambda< \frac{2k+1}{\omega_{1}}$ then there exist
two $p$ periodic solutions such that the solutions $x(t)=u(\lambda t)$ to
\eqref{eq:system2} has the isotropy groups $\bz_{2}\times D_{4m+2}^{d}$ and
$V_{4}^{\bz_{2}^{-}}\times_{\bz_{2}}D_{4m^{\prime}+2}$ for some $m$,
$m^{\prime}$.
\item[(c)] If $\frac{2k+1}{\omega_{3}}<\lambda< \frac{2k+1}{\omega_{2}}$  then there exist
two $p$ periodic solutions such that the solutions $x(t)=u(\lambda t)$ to
\eqref{eq:system2} has the isotropy group
$V_{4}^{\bz_{2}^{-}}\times_{\bz_{2}}D_{4m+2}$ for some $m$.
\end{itemize}
\end{theorem}

The periodic solutions with the isotropy groups of this theorem have symmetries:

\begin{itemize}
\item For $j$ being an even number, a non-zero function in $x$ with the orbit
type $(\bz_{2}\times D_{2k}^{d})$ has for all $t\in\br$
\[
x_{j}(t)=x_{j}(t+\textstyle{\frac{2\pi}{k}})=-x_{j}(t+\textstyle{\frac{\pi}
{k}})=x_{n+1-j}(t)=x_{j}(-t).
\]
These functions are $\frac{2\pi}{k}$-periodic and $\frac{\pi}{k}$-antiperiodic
functions which are symmetric with respect to time reversion, and symmetric
under the permutation $\xi(j)=n+1-j$

\item For $j$ being an odd number, a non-zero function in $x$ has the orbit
type\break$(V_{4}^{\bz_{2}^{-}}\times_{\bz_{2}}D_{2k})$. Therefore, we have
for all $t\in\br$
\[
x_{j}(t)=x_{j}(t+\textstyle{\frac{2\pi}{k}})=-x_{j}(t+\textstyle{\frac{\pi}
{k}})=-x_{n+1-j}(t)=x_{j}(-t).
\]
These functions are $\frac{2\pi}{k}$-periodic, $\frac{\pi}{k}$-antiperiodic
functions which are symmetric with respect to time reversion, and
anti-symmetric under the permutation $\xi(j)=n+1-j$.
\end{itemize}

\vs

\subsection{Cycle of $n$ elements with $\Gamma=D_{n}$-symmetries}

\begin{figure}[H]
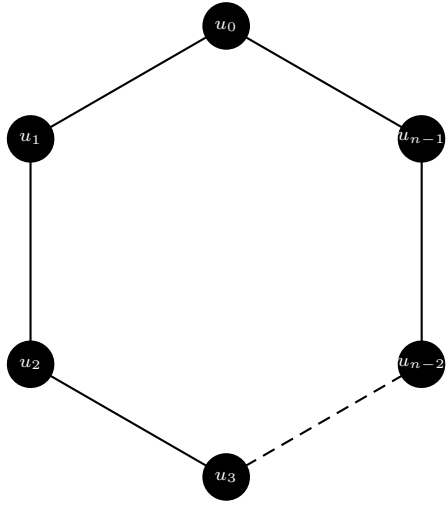

\begin{center}
\vskip3cm

\rput(0,0){\psline(2.6,-1.5)(2.6,1.5)}
\rput{60}(0,0){\psline(2.6,-1.5)(2.6,1.5)}
\rput{120}(0,0){\psline(2.6,-1.5)(2.6,1.5)}
\rput{180}(0,0){\psline(2.6,-1.5)(2.6,1.5)}
\rput{-60}(0,0){\psline[linestyle=dashed](2.6,-1.5)(2.6,1.5)}
\rput{-120}(0,0){\psline(2.6,-1.5)(2.6,1.5)}
\rput{60}(0,0){\rput(0,3){\psdots[linewidth=8pt](0,0)\rput{-60}(0,0){\scriptsize \white $u_1$}}}
\rput{0}(0,0){\rput(0,3){\psdots[linewidth=8pt](0,0)\rput(0,0){\scriptsize \white $u_0$}}}
\rput{120}(0,0){\rput(0,3){\psdots[linewidth=8pt](0,0)\rput{-120}(0,0){\scriptsize \white $u_2$}}}
\rput{-60}(0,0){\rput(0,3){\psdots[linewidth=8pt](0,0)\rput{60}(0,0){\scriptsize \white $u_{n-1}$}}}
\rput{-120}(0,0){\rput(0,3){\psdots[linewidth=8pt](0,0)\rput{120}(0,0){\scriptsize \white $u_{n-2}$}}}
\rput{180}(0,0){\rput(0,3){\psdots[linewidth=8pt](0,0)\rput{180}(0,0){\scriptsize \white $u_3$}}}

\end{center}
\vskip3cm
\caption{Cycle of $n$ elements with $\Gamma=D_{n}$-symmetries}\label{fig:chain}
\end{figure}

For the cycle the matrix $\Delta_{G}$ is given by
\[
\Delta_{G}=\left(
\begin{array}
[c]{ccccccc}
-2 & 1 & 0 &\dots  & 0 & 0 & 1\\
1 & -2 & 1 & \dots & 0 & 0 & 0\\
0&1 & -2 & \dots & 0 & 0 & 0\\
\vdots&\vdots  & \vdots &\ddots  & \vdots & \vdots & \vdots\\
0 & 0 & 0 & \dots & -2 & 1&0\\
0 & 0 & 0 & \dots & 1 & -2 & 1\\
1 & 0 & 0 & \dots & 0 & 1 & -2
\end{array}
\right)  .
\]
In this case the matrix $\Delta_{G}$ commutes with $\Gamma={D}_{n}
=\{1,\zeta,\dots,\zeta^{n-1},\xi,\xi\zeta,\dots,\xi\zeta^{n-1}\}$, where
$\zeta$, $\xi$ are the permutations of the symbols $\{0,1,2,\dots,n-1\}$ such
that $\zeta(j)=j+1$ and $\xi(j)=n-j$ (mod $n$), acting on $V:=\mathbb{R}^{n}$
by
\[
\gamma(x_{0},x_{1},\dots,x_{n-1})=(x_{\zeta(0)},x_{\zeta(1)},\dots
,x_{\zeta(n-1)}),\quad\xi(x_{0},x_{1},\dots,x_{n-1})=(x_{\xi(0)},x_{\xi
(1)},\dots,x_{\xi(n-1)}).
\]

The eigenvalues of $\triangle_{G}$ can be easily found. Consider
$\triangle_{G}:{\mathbb{C}}^{n}\rightarrow{\mathbb{C}}^{n}$ and put
\[v_{j}:=(1,\zeta^{j},\zeta^{2j},\dots,\zeta^{(n-1)j})^{T}\]
with $\zeta=e^{2\pi
i}$, $j=0,1,\dots,n-1.$ Clearly, $v_{j}$ is a complex eigenvector
corresponding to the eigenvalue
\[
\mu_{j}:=-2+\zeta^{j}+\zeta^{-j}=-4\sin^{2}\frac{\pi j}{n}.
\]
Therefore, the real part $e_{j}$ of $v_{j}$ is a real eigenvector of
$\triangle_{G}$, and for $0<j<n-1$, $j\not =\frac{n}{2}$, the imaginary part
$f_{j}$ of $v_{j}$ is another real eigenvector of $\triangle_{G}$. In the case
$j=0$ we have the eigenvector $e_{0}=(1,1,1,\dots,1)^{T}$ and if $r=\frac
{n}{2}$ is an integer, we also have $e_{r}=(1,-1,1,\dots,-1)^{T}$. Denote by
$V_{j}$, $j\not =0,\frac{n}{2}$, the subspace of $V$ spanned by $e_{j}$ and
$f_{j}$. Since $v_{j}=e_{j}+if_{j}$, $V_{j}$ has a natural complex structure
isomorphic to ${\mathbb{C}}$. One can easily recognize that $V_{j}$ is $D_{n}
$-invariant and equivalent to the irreducible $D_{n}$-representation
$\mathcal{V}_{j}\simeq{\mathbb{C}}$, where $\zeta z=\zeta^{j}\cdot z$ (here
`$\cdot$' denotes usual complex multiplication) and $\xi z:=\overline{z}$,
$z\in{\mathbb{C}}$. Since the $D_{n}$-representations $V_{j}=V_{n-j}$, for
$0<j<\frac{n}{2}$, $V_{j}$ is the $j$-th isotypical component of $V$. It is
clear that the $\mathcal{V}_{j}$-isotypical dimension of $V_{j}$ is $1$ for
$0\leq j\leq r:=\lfloor\frac{n}{2}\rfloor$, so we have the isotypical
multiplicities $m_{j}(\mu_{j})=1$.

As an example, we consider the case $n=4$. The character table for the related
irreducible $D_{4}\times\bz_{2}$-representations are listed in Table
\ref{tab:D4-Ch}. \begin{table}[H]
\centering
\begin{tabular}
[c]{l|rrrrrrrrrr}
\toprule & $(1,1)$ & $(\kappa,1)$ & $(i,1)$ & $(-1,1)$ & $(1,-1)$ & $(\kappa
i,1)$ & $(\kappa,-1)$ & $(i,-1)$ & $(-i,-1)$ & $(\kappa i, -1)$\\
\midrule ${\mathcal{V}_{0}^{-}}$ & $1$ & $1$ & $1$ & $1$ & $-1$ & $1$ & $-1$ &
$-1$ & $-1$ & $-1$\\
${\mathcal{V}_{1}^{-}}$ & $2$ & $0$ & $0$ & $-2$ & $-2$ & $0$ & $0$ & $0$ &
$2$ & $0$\\
${\mathcal{V}_{2}^{-}}$ & $1$ & $-1$ & $-1$ & $1$ & $1$ & $1$ & $-1$ & $-1$ &
$1$ & $1$\\
\bottomrule 
\end{tabular}
\vskip 1em \caption{character table of $D_{4}\times\bz_{2}$}
\label{tab:D4-Ch}
\end{table}We have the following $D_{4}\times\bz_{2}\times O(2)$-basic degrees
for $k=1,2,3...$:
\begin{align*}
\deg_{{\mathcal{V}_{1,k}^{-}}}=\;  &
(\amal{D_4^p}{O(2)}{}{}{})-(\amal{\tD_2^p}{D_{2k}}{\bz_{2}}{\tD_2^\td}{})-(\amal{D_4^p}{D_{4k}}{D_{4}}{\bz_2^z}{})-(\amal{D_2^p}{D_{2k}}{\bz_{2}}{D_2^d}{})\\
&
+(\amal{D_2^p}{D_{2k}}{D_{2}}{\bz_2^z}{\bz_2^p})+(\amal{\bz_2^p}{D_{2k}}{\bz_{2}}{\bz_2^z}{})+(\amal{\tD_2^p}{D_{2k}}{D_{2}}{\bz_2^z}{\bz_2^p}),\\
\deg_{{\mathcal{V}_{2,k}^{-}}}=\;  &
(\amal{D_4^p}{O(2)}{}{}{})-(\amal{D_4^p}{D_{2k}}{\bz_{2}}{\tD_2^p}{}).
\end{align*}

The values $k/\omega_{j}$ for $k\in2\mathbb{N}+1$ and $j=1,2$ (with
$\omega_{1}=\sqrt{2}$ and $\omega_{2}=\allowbreak2$) have the order
\[
\frac{1}{\omega_{2}}<\frac{1}{\omega_{1}}<\frac{3}{\omega_{2}}<\frac{3}
{\omega_{1}}<...<\frac{2k+1}{\omega_{2}}<\frac{2k+1}{\omega_{1}}<\dots
\]
Therefore, we have the following result: \vs

\begin{theorem}
\label{th:chain copy(1)} Consider the network consisting of a chain of four
elements and assume that $g^{\prime}(0)=0$ and $\left\vert g(u)\right\vert
>\left\vert u\right\vert $ for $\left\vert u\right\vert >M$. Suppose
$\lambda>\frac{1}{\omega_{2}}$ such that for all $k\in2\mathbb{N}-1$ we
have $\lambda\not =\frac{k}{\omega_{j}}$, $j=1,2$. Then the system
\eqref{eq:system1} admits a $p$-periodic solution $u(t)$. Moreover,

\begin{itemize}
\item[(a)] If $\frac{1}{\omega_{2}}<\lambda<\frac{1}{\omega_{1}}$ then there
exists an orbit of  $p$ periodic solutions $u(t)$ such that the solution $x(t)=u(\lambda
t)$ to \eqref{eq:system2} has the isotropy group
$\amal{D_4^p}{D_{2m}}{\bz_{2}}{\tD_2^p}{}$  for some $m\in \bn$.

\item[(b)] If $\frac{1}{\omega_{1}}<\lambda$ then there
exist four orbits of  $p$ periodic solutions such that the solutions $x(t)=u(\lambda t)$
to \eqref{eq:system2} have the isotropy groups
$\amal{D_4^p}{D_{2m}}{\bz_{2}}{\tD_2^p}{}$,  $\amal{D_4^p}{D_{4m}}{D_{4}}{\bz_2^z}{}$, $\amal{D_2^p}{D_{2m}}{\bz_{2}}{D_2^d}{}$, and $\amal{D_4^p}{D_{2m}}{\bz_{2}}{\tD_2^p}{}$  for some $m\in \bn$,
\end{itemize}
\end{theorem}

\vs

\subsection{Truncated octahedron graph with $\Gamma=S_{4}$ octahedral
symmetries}

\begin{figure}[t]
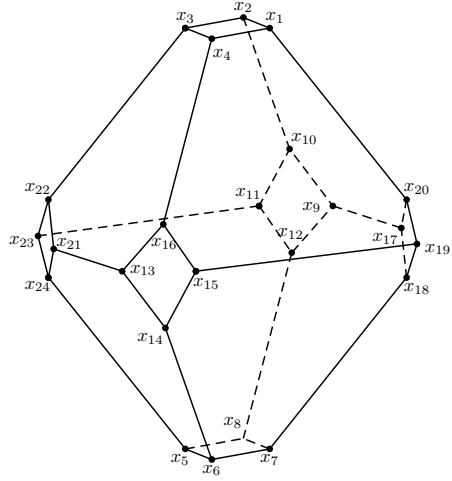

\vglue3.cm\hskip7cm
\scalebox{.7}{\psline(-.8,4)(-3.4,.74)
\psline(-3.4,-.74)(-.8,-4)
\psline(.8,-4)(3.4,-.74)
\psline(.8,4)(3.4,.74)
\psline(-3.3,-0.2)(-2,-.62)
\psline(-.6,-.62)(3.6,-0.1)
\psline[linestyle=dashed](3.3,0.2)(2,.62)
\psline[linestyle=dashed](.6,.62)(-3.6,0.05)
\psline(-.3,3.8)(-1.22,.27)
\psline(-1.18,-1.7)(-.3,-4.2)
\psline[linestyle=dashed](.3,4.2)(1.18,1.7)
\psline[linestyle=dashed](1.22,-.27)(.3,-3.8)
\psline(-.8,4)(.3,4.2)(.8,4)(-.3,3.8)(-.8,4)
\psline(-.8,-4)(-.3,-4.2)(.8,-4)
\psline[linestyle=dashed](.8,-4)(.3,-3.8)(-.8,-4)
\psline(3.4,-.74)(3.6,-0.1)(3.4,.74)
\psline[linestyle=dashed](3.4,.74)(3.3,0.2)(3.4,-.74)
\psline(-3.4,-.74)(-3.3,-0.2)(-3.4,.74)(-3.6,0.05)(-3.4,-.74)
\psline(-.6,-.62)(-1.22,.27)(-2,-.62)(-1.18,-1.7)(-.6,-.62)
\psline[linestyle=dashed](.6,.62)(1.22,-.27)(2,.62)(1.18,1.7)(.6,.62)
\psdots
(.6,.62)(1.22,-.27)(2,.62)(1.18,1.7)(.6,.62) (-.6,-.62)(-1.22,.27)(-2,-.62)(-1.18,-1.7)(-.6,-.62)(3.4,.74)(3.3,0.2)(3.4,-.74)(-.8,4)(.3,4.2)(.8,4)(-.3,3.8)(-.8,4)(-.8,-4)(-.3,-4.2)(.8,-4)(-3.4,-.74)(-3.3,-0.2)(-3.4,.74)(-3.6,0.05)(-3.4,-.74)(3.6,-0.1)
\rput(-.8,4.2){$x_3$}
\rput(.3,4.4){$x_2$}
\rput(.9,4.2){$x_1$}
\rput(-.1,3.5){$x_4$}
\rput(.8,-4.2){$x_7$}
\rput(-.3,-4.4){$x_6$}
\rput(-.9,-4.2){$x_5$}
\rput(.1,-3.5){$x_8$}
\rput(3.6,.95){$x_{20}$}
\rput(3,0){$x_{17}$}
\rput(3.6,-.95){$x_{18}$}
\rput(4,-.2){$x_{19}$}
\rput(-3.6,-.9){$x_{24}$}
\rput(-3.9,-.05){$x_{23}$}
\rput(-3,-.1){$x_{21}$}
\rput(-3.6,.95){$x_{22}$}
\rput(1.46,1.9){$x_{10}$}
\rput(1.2,.1){$x_{12}$}
\rput(1.6,.6){$x_{9}$}
\rput(.4,.85){$x_{11}$}
\rput(-1.46,-1.9){$x_{14}$}
\rput(-1.2,-.1){$x_{16}$}
\rput(-1.6,-.6){$x_{13}$}
\rput(-.4,-.85){$x_{15}$}
}
\vskip3cm
\caption{Truncated octaheron graph with $\Gamma=S_{4}$ octahedral symmetries}\label{fig:octahedron}
\end{figure}

We can index the vertices of the cut-off octahedron by the  elements in $ (\sigma
,\tau)\in S_4\times S_4$, where $\sigma$ is a 2-cycle and $\tau
$ is a 4-cycle. The full conversion table corresponding to the annotated above model is given below:
\begin{alignat*}{3}
1\;&\leftrightarrow\; ((1,2),(1,2,3,4)),&\quad\; 5\;&\leftrightarrow
\; ((1,2),(1,4,3,2))\\
2\;&\leftrightarrow\; ((2,3),(1,2,3,4)),&\quad\; 6\;&\leftrightarrow
\; ((2,3),(1,4,3,2))\\
3\;&\leftrightarrow\; ((3,4),(1,2,3,4)),&\quad\; 7\;&\leftrightarrow
\; ((3,4),(1,4,3,2))\\
\psline(0,-.2)(8,-.2)
4\;&\leftrightarrow\; ((1,4),(1,2,3,4)),&\quad\; 8\;&\leftrightarrow
\; ((1,4),(1,4,3,2))\\
9\;&\leftrightarrow\; ((1,3),(1,3,2,4)),&\quad\; 13\;&\leftrightarrow
\; ((1,3),(1,4,2,3))\\
10\;&\leftrightarrow\; ((2,3),(1,3,2,4)),&\quad\; 14\;&\leftrightarrow
\; ((2,3),(1,4,2,3))\\
11\;&\leftrightarrow\; ((2,4),(1,3,2,4)),&\quad\; 15\;&\leftrightarrow
\; ((2,4),(1,4,2,3))\\
\psline(0,-.2)(8,-.2)
12\;&\leftrightarrow\; ((1,4),(1,3,2,4)),&\quad\; 16\;&\leftrightarrow
\; ((1,4),(1,4,2,3))\\
17\;&\leftrightarrow\; ((1,3),(1,3,4,2)),&\quad\; 21\;&\leftrightarrow
\; ((1,3),(1,2,4,3))\\
18\;&\leftrightarrow\; ((3,4),(1,3,4,2)),&\quad\; 22\;&\leftrightarrow
\; ((3,4),(1,2,4,3))\\
19\;&\leftrightarrow\; ((2,4),(1,3,4,2)),&\quad\; 23\;&\leftrightarrow
\; ((2,4),(1,2,4,3))\\
20\;&\leftrightarrow\; ((1,2),(1,3,4,2)),&\quad\; 24\;&\leftrightarrow
\; ((1,2),(1,2,4,3))
\end{alignat*}
The action of the full octahedral group $\mathbb O:=\bz_2\times
S_4)$ on the space $V:=\br^{24}$ can be easily described; for $\eta\in
S_4$ and $-1\in\bz_2$, we have for
$x=(x_1,x_2,\dots,x_{24})$, $x_j=:x_{\sigma,tau}$
\[
\eta(x_{\sigma,\tau}):=x_{\eta^{-1}\sigma\eta,\eta\tau\eta^{-1}}
, \quad(-1)x_{\sigma,\tau}:=-x_{\sigma,\tau}=x_{\sigma,\tau^{-1}}.
\]

\vskip.3cm
In order to identify the isotypical decomposition of $V$ first we consider the following character table:
\begin{center}
\begin{tabular}{|c|c|ccccc|}
\hline
Iso Dim.& $\chi_j$ & $()$& $(12)$ &$(12)(34)$ &$(123)$ &$(1234)$\\
\hline
 1&  $\chi_0$& 1 & 1& 1& 1 &1\\	
 1&$\chi_1$&1&$-1$&1&1&$-1$\\
 2&$\chi_2$&2&0&2&$-1$&0\\
 3&$\chi_3$&3&-1&$-1$&0&$1$\\
 3&$\chi_4$&3&$1$&$-1$&0&-1\\

 \hline
 \end{tabular}
\end{center}

Therefore, we have the following isotypical decomposition
\begin{gather*}
V=V_{0}\oplus V_{1}\oplus V_{2}\oplus V_{3}\oplus V_{4},\quad V_{0}
=\mathcal{V}_{0},\;V_{1}=\mathcal{V}_{1},\;V_{2}=\mathcal{V}_{2}
\oplus\mathcal{V}_{2},\\
V_{3}=\mathcal{V}_{3}\oplus\mathcal{V}_{3}\oplus\mathcal{V}_{3},\;V_{4}
=\mathcal{V}_{4}\oplus\mathcal{V}_{4}\oplus\mathcal{V}_{4}\text{.}
\end{gather*}

For this cube-like graph, the matrix $\triangle_{G}$ is given by
\vs\scalebox{.7}{
$\triangle_G=\left[\begin{array}
[c]{cccccccccccccccccccccccc}
-3 & 1 & 0 & 1 & 0 & 0 & 0 & 0 & 0 & 0 & 0 & 0 & 0 & 0 & 0 & 0 & 0 & 0 &
1 & 0 & 0 & 0 & 0\\
1 & -3 & 1 & 0 & 0 & 0 & 0 & 0 & 0 & 1 & 0 & 0 & 0 & 0 & 0 & 0 & 0 & 0 &
0 & 0 & 0 & 0 & 0 & 0\\
0 & 1 & -3 & 1 & 0 & 0 & 0 & 0 & 0 & 0 & 0 & 0 & 0 & 0 & 0 & 0 & 0 & 0 &
0 & 0 & 0 & 1 & 0 & 0\\
1 & 0 & 1 & -3 & 0 & 0 & 0 & 0 & 0 & 0 & 0 & 0 & 0 & 0 & 0 & 1 & 0 & 0 &
0 & 0 & 0 & 0 & 0 & 0\\
0 & 0 & 0 & 0 & -3 & 1 & 0 & 1 & 0 & 0 & 0 & 0 & 0 & 0 & 0 & 0 & 0 & 0 &
0 & 0 & 0 & 0 & 0 & 1\\
0 & 0 & 0 & 0 & 1 & -3 & 1 & 0 & 0 & 0 & 0 & 0 & 0 & 1 & 0 & 0 & 0 & 0 &
0 & 0 & 0 & 0 & 0 & 0\\
0 & 0 & 0 & 0 & 0 & 1 & -3 & 1 & 0 & 0 & 0 & 0 & 0 & 0 & 0 & 0 & 0 & 1 &
0 & 0 & 0 & 0 & 0 & 0\\
0 & 0 & 0 & 0 & 1 & 0 & 1 & -3 & 0 & 0 & 0 & 1 & 0 & 0 & 0 & 0 & 0 & 0 &
0 & 0 & 0 & 0 & 0 & 0\\
0 & 0 & 0 & 0 & 0 & 0 & 0 & 0 & -3 & 1 & 0 & 1 & 0 & 0 & 0 & 0 & 1 & 0 &
0 & 0 & 0 & 0 & 0 & 0\\
0 & 1 & 0 & 0 & 0 & 0 & 0 & 0 & 1 & -3 & 1 & 0 & 0 & 0 & 0 & 0 & 0 & 0 &
0 & 0 & 0 & 0 & 0 & 0\\
0 & 0 & 0 & 0 & 0 & 0 & 0 & 0 & 0 & 1 & -3 & 1 & 0 & 0 & 0 & 0 & 0 & 0 &
0 & 0 & 0 & 0 & 1 & 0\\
0 & 0 & 0 & 0 & 0 & 0 & 0 & 1 & 1 & 0 & 1 & -3 & 0 & 0 & 0 & 0 & 0 & 0 &
0 & 0 & 0 & 0 & 0 & 0\\
0 & 0 & 0 & 0 & 0 & 0 & 0 & 0 & 0 & 0 & 0 & 0 & -3 & 1 & 0 & 1 & 0 & 0 &
0 & 0 & 1 & 0 & 0 & 0\\
0 & 0 & 0 & 0 & 0 & 1 & 0 & 0 & 0 & 0 & 0 & 0 & 1 & -3 & 1 & 0 & 0 & 0 &
0 & 0 & 0 & 0 & 0 & 0\\
0 & 0 & 0 & 0 & 0 & 0 & 0 & 0 & 0 & 0 & 0 & 0 & 0 & 1 & -3 & 1 & 0 & 0 &
1 & 0 & 0 & 0 & 0 & 0\\
0 & 0 & 0 & 1 & 0 & 0 & 0 & 0 & 0 & 0 & 0 & 0 & 1 & 0 & 1 & -3 & 0 & 0 &
0 & 0 & 0 & 0 & 0 & 0\\
0 & 0 & 0 & 0 & 0 & 0 & 0 & 0 & 1 & 0 & 0 & 0 & 0 & 0 & 0 & 0 & -3 & 1 &
0 & 1 & 0 & 0 & 0 & 0\\
0 & 0 & 0 & 0 & 0 & 0 & 1 & 0 & 0 & 0 & 0 & 0 & 0 & 0 & 0 & 0 & 1 & -3 &
1 & 0 & 0 & 0 & 0 & 0\\
0 & 0 & 0 & 0 & 0 & 0 & 0 & 0 & 0 & 0 & 0 & 0 & 0 & 0 & 1 & 0 & 0 & 1 &
-3 & 1 & 0 & 0 & 0 & 0\\
1 & 0 & 0 & 0 & 0 & 0 & 0 & 0 & 0 & 0 & 0 & 0 & 0 & 0 & 0 & 0 & 1 & 0 &
1 & -3 & 0 & 0 & 0 & 0\\
0 & 0 & 0 & 0 & 0 & 0 & 0 & 0 & 0 & 0 & 0 & 0 & 1 & 0 & 0 & 0 & 0 & 0 &
0 & 0 & -3 & 1 & 0 & 1\\
0 & 0 & 1 & 0 & 0 & 0 & 0 & 0 & 0 & 0 & 0 & 0 & 0 & 0 & 0 & 0 & 0 & 0 &
0 & 0 & 1 & -3 & 1 & 0\\
0 & 0 & 0 & 0 & 0 & 0 & 0 & 0 & 0 & 0 & 1 & 0 & 0 & 0 & 0 & 0 & 0 & 0 &
0 & 0 & 0 & 1 & -3 & 1\\
0 & 0 & 0 & 0 & 1 & 0 & 0 & 0 & 0 & 0 & 0 & 0 & 0 & 0 & 0 & 0 & 0 & 0 &
0 & 0 & 1 & 0 & 1 & -3
\end{array}\right]
$} The matrix $\triangle_{G}$ has the following eigenvalues $\mu_{j}$ and
eigenvectors $v^{j,i}\in\mathcal{V}_{j}$ such that $\mathcal{V}_{j}
=$span\thinspace$\{v_{j,i}:i=1,...,\dim\mathcal{V}_{j}\}$:

\begin{itemize}
\item $\mu_{0}:=0$ with eigenspace $E(\mu_{0})$ equivalent to $\mathcal{V}
_{0}^{-}$ spanned by the vector:\newline\noi\scalebox{.78}{
$\begin{array}{cl}
v_{10}&=(1,1,1,1,1,1,1,1,1,1,1,1,1,1,1,1,1,1,1,1,1,1,1,1)^T
\end{array}$}

\item $\mu_{1}:=\sqrt{2}-2$ with eigenspace $E(\mu_{1})$ equivalent to
$\mathcal{V}_{3}^{-}$ spanned by the vectors:\newline
\noi\scalebox{.78}{$\begin{array}{cl} v_{1,1}&=[0,1,0,-1,0,-1,0,1,1+\sqrt 2  ,1+\sqrt 2 ,1+\sqrt 2,1+\sqrt 2,-1-\sqrt{2},-1-\sqrt 2 ,-1-\sqrt{2},-1-\sqrt 2 ,1,0,-1,0,-1,0,1,0]^T\\
v_{1,2}&=[1,\frac{\sqrt 2}2,-1-\frac {\sqrt 2}2 ,\sqrt{2},-1,-\frac{\sqrt 2}2,-\sqrt{2},1+\frac {\sqrt 2}2, -\sqrt 2 ,-\frac{\sqrt 2}2,-1,1+\frac {\sqrt 2}2 ,\sqrt{2},\frac{\sqrt 2}2,1,\frac{\sqrt 2}2,-1,-1,0,0,1,1,0,0]^T\\
v_{1,3}&=[-\sqrt 2,1+\frac {\sqrt 2}2,-1,-\frac{\sqrt 2}2,\sqrt 2 ,-1-\frac {\sqrt 2}2 ,1,\frac{\sqrt 2}2,-\sqrt 2,1+\frac {\sqrt 2}2 ,-1,-\frac{\sqrt 2}2,\sqrt{2},-1-\frac {\sqrt 2}2,1,-1-\frac {\sqrt 2}2 ,
-1,0,0,-1,1,0,0,1]^T\end{array}$
} \newline

\item $\mu_{2}:=\sqrt{3}-3$ with eigenspace $E(\mu_{2})$ equivalent to
$\mathcal{V}_{2}^{-}$ spanned by the vectors:\newline\noi\scalebox{.78}{
$\begin{array}{cl}
v_{7,1}&=-2,-\sqrt{3},-2,-\sqrt{3},-2,-\sqrt{3},-2,-\sqrt{3},\sqrt{3},1,\sqrt{3},1,\sqrt{3},1,\sqrt{3},1,1,0,1,0,1,0,1,0)^T\\
v_{7,2}&=(\sqrt{3},1,\sqrt{3},1,\sqrt{3},1,\sqrt{3},1,-2,-\sqrt{3},-2,-\sqrt{3},-2,-\sqrt{3},-2,-\sqrt{3},0,1,0,1,0,1,0,1)^T
\end{array}$}

\item $\mu_{3}=-2$ with eigenspace $E(\mu_{3})$ equivalent to $\mathcal{V}
_{4}^{-}$ spanned by the vectors:\newline
\noi\scalebox{.78}{$\begin{array}{cl}
v_{5,1}&=(0,-1,0,1,0,-1,0,1,0,-1,0,1,0,-1,0,1,0,0,0,0,0,0,0,0)^T\\
v_{5,2}&=(1,0,-1,0,1,0,-1,0,0,0,0,0,0,0,0,0,0,-1,0,1,0,-1,0,1)^T\\
v_{5,3}&=(0,0,0,0,0,0,0,0,-1,0,1,0,-1,0,1,0,-1,0,1,0,-1,0,1,0)^T
\end{array}$}

\item $\mu_{4}=\sqrt{2}-4$ with eigenspace $E(\mu_{4})$ equivalent to
$\mathcal{V}_{4}^{-}$ spanned by the vectors:\newline
\noi\scalebox{.78}{$\begin{array}{cl}
v_{4,1}&=(-1,-\frac {\sqrt 2}2,\sqrt{2},\frac {2-\sqrt 2}2,1,\frac {\sqrt 2}2,-\sqrt{2},-\frac {2-\sqrt 2}2,\sqrt 2,-\frac {\sqrt 2}2,-1,\frac {2-\sqrt 2}2,-\sqrt{2},\frac {\sqrt 2}2,1,-\frac {2-\sqrt 2}2 ,1,-1,0,0,-1,1,0,0)^T\\
v_{4,2}&=(0,-1,0,1,0,1,0,-1,-1+\sqrt 2 ,1-\sqrt 2,1+\sqrt 2,1-\sqrt 2,1-\sqrt{2},-1+\sqrt 2,1-\sqrt{2},-1+\sqrt 2,1,0,-1,0,-1,0,1,0)^T\\
v_{4,3}&=(
-\sqrt 2,-\frac {2-\sqrt 2}2 ,1,\frac {\sqrt 2}2,\sqrt 2,
\frac {2-\sqrt 2}2,-1,-\frac {\sqrt 2}2,\sqrt 2,\frac {2-\sqrt 2}2,-1,-\frac {\sqrt 2}2,-\sqrt{2},-\frac {2-\sqrt 2}2,1,\frac {\sqrt 2}2,1,0,0,-1,-1,0,0,1)^T.\end{array}$}

\item $\mu_{5}=-\sqrt{2}-2,$ with eigenspace $E(\mu_{5})$ equivalent to
$\mathcal{V}_{3}^{-}$ spanned by the vectors:\newline
\noi\scalebox{.78}{$\begin{array}{cl}
v_{2,1}=&[1,-\frac {\sqrt 2}2,-\sqrt{2},\frac {2-\sqrt 2}2 ,-1,\frac {\sqrt 2}2,\sqrt{2},-\frac {2-\sqrt 2}2,\sqrt 2,\frac {\sqrt 2}2,-1,-\frac {2-\sqrt 2}2,-\sqrt{2},-\frac {\sqrt 2}2,1,\frac {2-\sqrt 2}2,-1,-1,0,0,1,1,0,0]^T\\
v_{2,2}=&[0,1,0,-1,0,-1,0,1,1-\sqrt 2,1-\sqrt 2,1-\sqrt 2,1-\sqrt 2,\sqrt{2}-1,\sqrt 2-1,\sqrt{2}-1,\sqrt 2-1,1,0,-1,0,-1,0,1,0]^T\\
v_{2,3}=&[\sqrt 2,-\frac {2-\sqrt 2}2,-1,\frac {\sqrt 2}2,-\sqrt 2,\frac {2-\sqrt 2}2 ,1,-\frac {\sqrt 2}2,\sqrt 2,-\frac {2-\sqrt 2}2 ,-1,\frac {\sqrt 2}2,-\sqrt{2},
\frac {2-\sqrt 2}2,1,-\frac {\sqrt 2}2,-1,0,0,-1,1,0,0,1]^T
\end{array}
$}\newline

\item $\mu_{6}=-4$ with eigenspace $E(\mu_{6})$ equivalent to $\mathcal{V}
_{3}^{-}$ spanned by the vectors:\newline\noi\scalebox{.78}{
$\begin{array}{cl}
v_{6,1}&=(0,1,0,-1,0,1,0,-1,0,-1,0,1,0,-1,0,1,0,0,0,0,0,0,0,0)^T\\
v_{6,2}&=(-1,0,1,0,-1,0,1,0,0,0,0,0,0,0,0,0,0,-1,0,1,0,-1,0,1)^T\\
v_{6,3}&=(0,0,0,0,0,0,0,0,1,0,-1,0,1,0,-1,0,-1,0,1,0,-1,0,1,0)^T
\end{array}$}

\item $\mu_{7}:=-\sqrt{3}-3$ with eigenspace $E(\mu_{7})$ equivalent to
$\mathcal{V}_{2}^{-}$ spanned by the vectors:\newline\noi\scalebox{.78}{
$\begin{array}{cl}
v_{8,1}&=(-2,\sqrt{3},-2,\sqrt{3},-2,\sqrt{3},-2,\sqrt{3},-\sqrt{3},1,-\sqrt{3},1,-\sqrt{3},1,-\sqrt{3},1,1,0,1,0,1,0,1,0)^T\\
v_{8,2}&=(-\sqrt{3},1,-\sqrt{3},1,-\sqrt{3},1,-\sqrt{3},1,-2,\sqrt{3},-2,\sqrt{3},-2,\sqrt{3},-2,\sqrt{3},0,1,0,1,0,1,0,1)^T
\end{array}$}

\item $\mu_{8}=-\sqrt{2}-4,$ with eigenspace $E(\mu_{8})$ equivalent to
$\mathcal{V}_{4}^{-}$ spanned by the vectors:\newline
\noi\scalebox{.78}{$\begin{array}{cl}
v_{3,1}=&[\sqrt 2,1+\frac {\sqrt 2}2,1,-\frac{\sqrt 2}2,-\sqrt 2,-1-\frac {\sqrt 2}2,-1,\frac{\sqrt 2}2,-\sqrt 2,-1-\frac {\sqrt 2}2, -1,\frac{\sqrt 2}2,\sqrt{2},1+\frac {\sqrt 2}2,1,-\frac{\sqrt 2}2,1,0,0,-1,-1,0,0,1]^T\\
v_{3,2}=&[-1,\frac{\sqrt 2}2,-\sqrt{2},-1-\frac {\sqrt 2}2,1,-\frac{\sqrt 2}2,\sqrt{2},1+\frac {\sqrt 2}2,-\sqrt 2,\frac{\sqrt 2}2,-1,-1-\frac {\sqrt 2}2),\sqrt{2},-\frac{\sqrt 2}2,1,1+\frac {\sqrt 2}2,1,-1,0,0,-1,1,0,0]^T\\
v_{3,3}=&[0,-1,0,1,0,1,0,-1,-1-\sqrt 2 ,1+\sqrt 2,-1-\sqrt 2,1+\sqrt 2,1+\sqrt{2},-1-\sqrt 2,1+\sqrt{2},-1-\sqrt 2,1,0,-1,0,-1,0,1,0]^T
\end{array}
$}

\item $\mu_{9}:=-6$ with eigenspace $E(\mu_{9})$ equivalent to $\mathcal{V}
_{1}^{-}$ spanned by the vector:\newline\noi\scalebox{.78}{
$\begin{array}{cl}
v_9&=(-1,1,-1,1,-1,1,-1,1,1,-1,1,-1,1,-1,1,-1,-1,1,-1,1,-1,1,-1,1)^T
\end{array}$}
\end{itemize}

In summary, we have the eigenvalues $\mu_{9}<...<\mu_{0}$ of $\triangle_{G}$,
each of them with single isotypical multiplicity (which means that the
eigenspace of the eigenvalue $\mu_{j}$ is equivalent to $\mathcal{V}_{n_{j}
}^{-}$). The values $k/\omega_{j}$ for $k\in\mathbb{N}$ and $j=1,...,9$ (with
$\mu_{j}=-\omega_{j}^{2}$) have the order
\[
\frac{1}{\omega_{9}}<....<\frac{1}{\omega_{3}}<\frac{2}{\omega_{9}}<\frac
{2}{\omega_{8}}<\frac{1}{\omega_{2}}<\frac{2}{\omega_{7}}<\frac{2}{\omega_{6}
}<\frac{2}{\omega_{5}}<\frac{1}{\omega_{1}}<\dots
\]

We apply the reduction to the fixed-point space of $H:=D_{1}^{z}$. In this
case $W(H)=S_{4}\times\bz_{2}$ and we have the following basic $S_{4}
\times\bz_{2}$-degrees:
\begin{align*}
\deg_{\cV_{0}^{-}}=\;  &  (S_{4}\times\bz_{2})-(S_{4}),\\
\deg_{\cV_{1}^{-}}=\;  &  (S_{4}\times\bz_{2})-(S_{4}^{-}),\\
\deg_{\cV_{2}^{-}}=\;  &  (S_{4}\times\bz_{2})-(D_{4}^{\hat{d}})-(D_{4}
)+(V_{4}),\\
\deg_{\cV_{3}^{-}}=\;  &  (S_{4}\times\bz_{2})-(D_{4}^{z})-(D_{3}^{z}
)-(D_{2}^{d})+2(D_{1}^{z})+(\bz_{2}^{-})-(\bz_{1}).\\
\deg_{\cV_{4}^{-}}=\;  &  (S_{4}\times\bz_{2})-(D_{4}^{d})-(D_{3})-(D_{2}
^{d})+(\bz_{2}^{-})+2(D_{1})-(\bz_{1}).
\end{align*}
By applying Theorem \ref{thm:exist2} and Corollary \ref{cor:prod}, we obtain the following result: \vs

\begin{theorem}
\label{th:octahedron} Consider the network consisting of a configuration of 20
elements arranged as the vertices of the cut-off octahedron and assume that
$g^{\prime}(0)=0$ and $\left\vert g(u)\right\vert >\left\vert u\right\vert $
for $\left\vert u\right\vert >M$. Suppose $\lambda>\frac{1}{\omega_{9}}$
such that for all $k\in2\mathbb{N}-1$ we have $\lambda\not =\frac{k}
{\omega_{j}}$, $j=1,...,9$. Then the system \eqref{eq:system1} admits a
$p$-periodic solution $u(t)$ satisfies $-u(-t)=u(t)$. Moreover,

\begin{itemize}
\item[(a)] If $\frac{1}{\omega_{9}}<\lambda<\frac{1}{\omega_{8}}$  then there exists a $p$ periodic
solution $u(t)$ such that the solution $x(t)=u(\lambda t)$ to
\eqref{eq:system2} has the $S_{4}\times\bz_{2}$-isotropy group $(S_{4}^{-})$.

\item[(b)] If $\frac{1}{\omega_{8}}<\lambda<\frac{1}{\omega_{7}}$ or $\frac
{1}{\omega_{2}}<\lambda<\frac{2}{\omega_{7}}$  then there exist three $p$ periodic solutions such that the
solutions $x(t)=u(\lambda t)$ to \eqref{eq:system2} have the $S_{4}
\times\bz_{2}$-isotropy groups $(S_{4}^{-})$, $(D_{4}^{{d}})$ and $(D_2^d)$.

\item[(c)] If $\frac{1}{\omega_{7}}<\lambda<\frac{1}{\omega_{6}}$, $\frac{1}{\omega_{5}}<\lambda<\frac{1}{\omega_{4}}$,  $\frac{1}{\omega_{3}}<\lambda<\frac{2}{\omega_{9}}$, $\frac{2}{\omega_{8}}<\lambda<\frac{1}{\omega_{2}}$, $\frac{2}{\omega_{7}}<\lambda<\frac{2}{\omega_{6}}$ and   $\frac{2}{\omega_{5}}<\lambda<\frac{1}{\omega_{1}}$, then there
exist four $p$ periodic solutions $u(t)$ such that the solutions
$x(t)=u(\lambda t)$ to \eqref{eq:system2} have the $S_{4}\times\bz_{2}
$-isotropy groups $(S_{4}^{-})$, $(D_{4}^{\hat{d}})$, $(D_2^d)$ and $D_{4}^{d})$.

\item[(d)] If $\frac{1}{\omega_{6}}<\lambda<\frac{1}{\omega_{5}}$, $\frac{2}{\omega_{8}}<\lambda<\frac{2}{\omega_{7}}$ and $\frac{2}{\omega_{6}}<\lambda<\frac{2}{\omega_{5}}$,   then there
exist three $p$ periodic solutions such that the solutions $x(t)=u(\lambda t)$
to \eqref{eq:system2} have the $S_{4}\times\bz_{2}$-isotropy groups $(S_{4}
^{-})$, $(D_4^d)$ and $(D_{4}^{\hat{d}})$.

\item[(e)] If $\frac{1}{\omega_{4}}<\lambda<\frac{1}{\omega_{3}}$  then there
exist three $p$ periodic solutions such that the solutions $x(t)=u(\lambda t)$
to \eqref{eq:system2} have the $S_{4}\times\bz_{2}$-isotropy groups $(S_{4}
^{-})$, $(D_2^d)$ and $(D_{4}^{\hat{d}})$.

\item[(f)] If $\frac{2}{\omega_{9}}<\lambda<\frac{2}{\omega_{8}}$  then there
exist three $p$ periodic solutions such that the solutions $x(t)=u(\lambda t)$
to \eqref{eq:system2} have the $S_{4}\times\bz_{2}$-isotropy groups $(D_{4}
^{d})$, $(D_2^d)$ and $(D_{4}^{\hat{d}})$.

\end{itemize}
\end{theorem}

\vskip1cm
\appendix
\section{Appendix: Equivariant Brouwer Degree}
\label{sec:pre}

W assume $G$ is a compact Lie group. For a subgroup
$H\subset G$ (which is always assumed to be closed), denote by $N(H)$ the
normalizer of $H$ in $G$, by $W(H)=N(H)/H$ the Weyl group of $H$ in $G$, and
by $(H)$ the conjugacy class of $H$ in $G$. The set $\Phi(G)$ of all conjugacy
classes in $G$ admits a partial order defined as follows: $(H)\leq(K)$ if and
only if $gHg^{-1}\subset K$ for some $g\in G$. We will also put $\Phi
_{k}(G):=\left\{  (H)\in\Phi(G):\dim{W(H)}=k\right\}  $. \vskip.3cm

For a $G$-space $X$ and $x\in X$, $G_{x}:=\left\{  g\in G:gx=x\right\}  $ is
called the \emph{isotropy} of $x$ and $G(x):=\left\{  gx:g\in G\right\}  $ is
called the \emph{orbit} of $x$. One can easily verify that $G(x)\simeq
G/G_{x}$. Denote by $X/G$ the the set of all orbits in $X$ under the action of
$G$, which is called the \emph{orbit space} of $X$. Furthermore, we call
conjugacy class $(G_{x})$ the \emph{orbit type} of $x$ in $X$ and put
$\Phi(G;X):=\left\{  (H)\in\Phi(G):H=G_{x}\mbox{ for some }x\in X\right\}  $.
Also, for a $G$-space $X$ and a closed subgroup $H$ of $G$, $X^{H}:=\left\{
x\in X:G_{x}\supset H\right\}  $ is called $H$-fixed-point subspace of $X$. \vskip.3cm

Let $X$ and $Y$ be two $G$-spaces. A continuous map $f:X\to Y$ is said to be
\emph{equivariant} if $f(gx) = gf(x)$ for all $x \in X$ and $g \in G$. If the
$G$-action on $Y$ is trivial, i.e., $f(gx)=f(x)$ for all $x\in X$ and $g\in
G$, then $f$ is called \textit{invariant}. As is known (see, for instance,
\cite{tD} and \cite{AED}), for any subgroup $H<G$ and equivariant map $f:X\to
Y$, the map $f^{H}:X^{H}\to Y^{H}$, where $f^{H}:=f|_{X^{H}}$, is well-defined
and $W(H)$-equivariant. \vskip.3cm

Let $L\leq H\leq G$. We put $N(L,H):=\left\{  g\in G:gLg^{-1}\leq H\right\}
$. Then, if $(L)$, $(H)\in\Phi_{0}(G)$, the number
\begin{align*}
n(L,H)=\left|  N(L,H)/N(H)\right|  ,
\end{align*}
where $\left|  X\right|  $ stands for the cardinality of $X$, is well defined
and finite (see, for example, \cite{GolSchSt}). For more information about the
number $n(L,H)$ and its properties, we refer to \cite{AED}. \vskip.3cm

\subsection{Isotypical Decomposition of Finite-Dimensional Representations}

\label{subsec:G-represent}

As is well-known, any compact group admits only countably many non-equivalent
real irreducible representations. Therefore, given a compact
Lie group $\Gamma$, we always assume that a complete list of all real
 irreducible $\Gamma$-representations, denoted by
$\mathcal{V}_{i}$, $i=0,1,\dots$  is
available. Let $V$ (resp.~$U$) be a finite-dimensional real 
$\Gamma$-representation. Without loss of generality, $V$ can be
assumed orthogonal. Then, one can represent $V$  as
the direct sum
\begin{align}
\label{eq:Giso}V  &  =V_{0}\oplus V_{1}\oplus\dots\oplus V_{r},
\end{align}
which is called the \emph{$\Gamma$-isotypical decomposition of $V$}, where the isotypical component $V_{i}$ is
\emph{modeled} on the irreducible $\Gamma$-representation $\mathcal{V}_{i}$,
$i=0,1,\dots,r$, i.e., $V_{i}$
 contains all the irreducible subrepresentations of $V$
 which are equivalent to $\mathcal{V}_{i}$. Notice that for a $\Gamma$-equivariant linear map $A:V\to V$,
$A(V_{i})\subset V_{i}$, $i=0,1,2,\dots,r$. We will denote by $\sigma(A)$ the
spectrum of $A$ and for $\mu\in\sigma(A)$, $E(\mu)$ will stand for the
generalized eigenspace corresponding to $\mu$. Clearly, $E(\mu)$ is $\Gamma
$-invariant. Then we can put
\begin{equation}
\label{eq:iso-mult}m_{i}(\mu):=\dim\left(  V_{i}\cap E(\mu)\right)
/\dim\mathcal{V}_{i},
\end{equation}
and will call the number $m_{i}(\mu)$ the $\mathcal{V}_{i}$-multiplicity of
the eigenvalue $\mu$. \vskip.3cm

Given an orthogonal $\Gamma$-representation $V$, denote by $\text{GL}^{\Gamma
}(V)$ group of all $\Gamma$-equivariant linear invertible operators on $V$.
Then, the isotypical decomposition \eqref{eq:Giso} induces a decomposition of
$\gl^{\Gamma}(V)$:
\begin{equation}
\label{eq:GLG-decomp}\gl^{\Gamma}(V)=\bigoplus_{i=0}^{r}\gl^{\Gamma}(V_{i}).
\end{equation}
For every isotypical component $V_{i}$ in \eqref{eq:Giso}, one has
$\gl^{\Gamma}(V_{i})\simeq\gl(m_{i},\mathbb{F})$, where $m_{i}=\dim V_{i}
/\dim\mathcal{V} _{i}$ and $\mathbb{F} $ is a finite-dimensional division
algebra, i.e., $\mathbb{F} =\mathbb{R} $, $\mathbb{C} $ or $\mathbb{H} $,
depending on the type of the irreducible representation $\mathcal{V} _{i}$. \vskip.3cm

\subsection{Burnside Ring $A(G)$}

Let $G$ be a compact Lie group. Denote by $A(G):={\mathbb{Z}}[\Phi_{0}(G)]$
the free abelian group generated by $(H)\in\Phi_{0} (G)$, i.e., an element
$a\in A(G)$ is a finite sum
\begin{align*}
a=n_{1}(H_{1})+\dots+n_{m}(H_{m}),
\end{align*}
where $n_{i}\in\mathbb{Z} $ and $(H_{i})\in\Phi_{0}(G)$. In addition, one can
define an operation of \emph{multiplication} in $A(G)$ by
\begin{equation}
\label{eq:Burnside-mul}(H)\cdot(K)=\sum_{(L)\in\Phi_{0}(G)}n_{L}\,(L),
\end{equation}
where the integer $n_{L}$ represents the number of orbits of type $(L)$
contained in the space $G/H\times G/K$. In this way, $A(G)$ becomes a ring
with the unity $(G)$. The ring $A(G)$ is called the \emph{Burnside ring} of
$G$. By using the partial order on $\Phi_{0}(G)$, the multiplication table for
$A(G)$ can be effectively computed using a simple recurrence formula:
\begin{equation}
\label{eq:Burnside-multip}n_{L}=\frac{n(L,H)\left|  W(H)\right| n
(L,K)\left|  W(K)\right|  - \sum_{(\widetilde L)>(L)}{n(L,\widetilde L)
n_{\widetilde L}\left|  W(\widetilde L)\right|  }}{\left|  W(L)\right|  }.
\end{equation}
\vskip.3cm

\subsection{$G$-Equivariant Degree: Basic Properties and Recurrence Formula}

\label{subsec:G-degree}

Suppose that $V$ is an orthogonal $G$-representation and $f:V\to V$ a
continuous $G$-equivariant map. Consider an open bounded $G$-invariant set
$\Omega$. Then the $G$-map $f$ is called \emph{$\Omega$-admissible} if for all
$x\in\partial\Omega$, we have $f(x)\neq0$. In such a case, the pair
$(f,\Omega)$ is called a \emph{$G$-admissible pair} (in $V$). The set of all
possible $G$-pairs will be denoted by $\mathcal{M}^{G}$. \vskip.3cm

This subsection is to provide a practical ``definition'' of the $G$
-equivariant degree based on its properties that can be used as set of axioms,
which uniquely determines this $G$-degree (see \cite{AED} for all the
details): \vskip.3cm

\begin{theorem}
\label{thm:GpropDeg} There exists a unique map $G\mbox{\rm -}\deg:\mathcal{M}
^{G}\to A(G)$, which assigns to every admissible $G$-pair $(f,\Omega)$ an
element $G\mbox{\rm -}\deg(f,\Omega)\in A(G)$, called the \emph{$G$
-equivariant degree (or simply $G$-degree)} of $f$ on $\Omega$,
\begin{equation}
\label{eq:G-deg0}G\mbox{\rm -}\deg(f,\Omega)=\sum_{(H)\in\Phi_{0}(G)}
{n_{H}(H)}= n_{H_{1}}(H_{1})+\dots+n_{H_{m}}(H_{m}),
\end{equation}
satisfying the following properties:

\begin{itemize}
\item[\textrm{(G1)}] \textbf{(Existence)} If $G\mbox{\rm -}\deg(f,\Omega)\ne
0$, i.e., $n_{H_{i}}\neq0$ for some $i$ in \eqref{eq:G-deg0}, then there
exists $x\in\Omega$ such that $f(x)=0$ and $(G_{x})\geq(H_{i})$.

\item[\textrm{(G2)}] \textbf{(Additivity)} Let $\Omega_{1}$ and $\Omega_{2}$
be two disjoint open $G$-invariant subsets of $\Omega$ such that
$f^{-1}(0)\cap\Omega\subset\Omega_{1}\cup\Omega_{2}$. Then,
\begin{align*}
G\mbox{\rm -}\deg(f,\Omega)=G\mbox{\rm -}\deg(f,\Omega_{1})+G\mbox{\rm -}\deg
(f,\Omega_{2}).
\end{align*}

\item[\textrm{(G3)}] \textbf{(Homotopy)} If $h:[0,1]\times V\to V$ is an
$\Omega$-admissible $G$-homotopy, then
\begin{align*}
G\mbox{\rm -}\deg(h_{t},\Omega)=\mathrm{constant}.
\end{align*}

\item[\textrm{(G4)}] \textbf{(Normalization)} Let $\Omega$ be a $G$-invariant
open bounded neighborhood of $0$ in $V$. Then,
\begin{align*}
G\mbox{\rm -}\deg(\id,\Omega)=1\cdot(G).
\end{align*}

\item[\textrm{(G5)}] \textbf{(Multiplicativity)} For any $(f_{1},\Omega
_{1}),(f_{2},\Omega_{2})\in\mathcal{M} ^{G}$,
\begin{align*}
G\mbox{\rm -}\deg(f_{1}\times f_{2},\Omega_{1}\times\Omega_{2})=
G\mbox{\rm -}\deg(f_{1},\Omega_{1})\cdot G\mbox{\rm -}\deg(f_{2},\Omega_{2}),
\end{align*}
where the multiplication `$\cdot$' is taken in the Burnside ring $A(G )$.

\item[\textrm{(G6)}] \textbf{(Suspension)} If $W$ is an orthogonal
$G$-representation and $\mathscr B$ is an open bounded invariant neighborhood
of $0\in W$, then
\begin{align*}
G\mbox{\rm -}\deg(f\times\id_{W},\Omega\times\mathscr B)=G\mbox{\rm -}\deg
(f,\Omega).
\end{align*}

\item[\textrm{(G7)}] \textbf{(Recurrence Formula)} For an admissible $G$-pair
$(f,\Omega)$, the $G$-degree \eqref{eq:G-deg0} can be computed using the
following recurrence formula
\begin{equation}
\label{eq:RF-0}n_{H}=\frac{\deg(f^{H},\Omega^{H})- \sum_{(K)>(H)}{n_{K}\,
n(H,K)\, \left|  W(K)\right|  }}{\left|  W(H)\right|  },
\end{equation}
where $\left|  X\right|  $ stands for the number of elements in the set $X$
and $\deg(f^{H},\Omega^{H})$ is the Brouwer degree of the map $f^{H}
:=f|_{V^{H}}$ on the set $\Omega^{H}\subset V^{H}$.
\end{itemize}
\end{theorem}

\vskip.3cm

The $G$-equivariant degree can be generalized, in a standard way (see
\cite{AED}) to a $G$-equivariant Leray-Schauder degree in infinite-dimensional
Banach spaces. \vskip.3cm

\subsection{Computation of $\Gamma$-Equivariant Degree}

Put $B(V):=\left\{  x\in V:\left|  x\right|  <1\right\}  $. For each
irreducible $G$-representation $\mathcal{V} _{i}$, $i=0,1,2,\dots$, we define
\begin{align*}
\deg_{\mathcal{V}_{i}}:=G\mbox{\rm -}\deg(-\id,B(\mathcal{V} _{i})),
\end{align*}
and will call $\deg_{\mathcal{V} _{i}}$ the \emph{basic degree}.

Consider a $G$-equivariant linear isomorphism $T:V\to V$ and assume that $V$
has a $G$-isotypical decomposition \eqref{eq:Giso}. Then, by the
Multiplicativity property (G5),
\begin{equation}
\label{eq:prod-prop}G\mbox{\rm -}\deg(T,B(V))=\prod_{i=0}^{r}G\mbox{\rm -}\deg
(T_{i},B(V_{i}))= \prod_{i=0}^{r}\prod_{\mu\in\sigma_{-}(T)} \left(
\deg_{\mathcal{V} _{i}}\right)  ^{m_{i}(\mu)}
\end{equation}
where $T_{i}=T|_{V_{i}}$ and $\sigma_{-}(T)$ denotes the real negative
spectrum of $T$, i.e., $\sigma_{-}(T)=\left\{  \mu\in\sigma(T):\mu<0\right\}
$. \vskip.3cm

Notice that the basic degrees can be effectively computed. Indeed
\begin{align*}
\deg_{\mathcal{V} _{i}}=\sum_{(H)\in\Phi_{0}(G)}n_{H}(H),
\end{align*}
where the coefficients $n_{H}$ can be computed from the following recurrence
formula
\begin{equation}
\label{eq:bdeg-nL}n_{H}=\frac{(-1)^{\dim\mathcal{V} _{i}^{H}}- \sum
_{H<K}{n_{K}\, n(H,K)\, \left|  W(K)\right|  }}{\left|  W(H)\right|  }.
\end{equation}
\vskip.3cm

\textbf{Acknowledgment: }The first author C. Garc\'{\i}a-Azpeitia is supported by PAPIIT-UNAM grant IA105217. The second author W.Krawcewicz acknowledges the support from
National Science Foundation through grant DMS-1413223 and the support from Guangzhou University. The third author Y.Lv is supported by National Natural Science Foundation of China through grant 11701325.

\end{document}